\documentclass[12pt]{tac}

\usepackage{amssymb}
\usepackage{amsmath}
\usepackage{mathabx}
\usepackage[all,cmtip]{xy}

\newcommand{\drpullback}[1][dr]{\save*!/#1-4ex/#1:(-1,1)@^{|-}\restore}
\newdir{ >}{{}*!/-7pt/@{>}}

\newtheorem{thm}{Theorem}
\newtheorem{prop}{Proposition}
\newtheorem{lemma}{Lemma}
\newtheorem{cor}{Corollary}

\providecommand{\qed}{\hspace*{\fill}\endproofbox}

\newenvironment{blanko}[1]{\subsection{#1}}{\penalty-700 \par\addvspace{\medskipamount}}
\newenvironment{thm*}{\subsection*{Theorem} \it}{\penalty-700 \par\addvspace{\medskipamount}}
\newenvironment{prop*}{\subsection*{Proposition} \it}{\penalty-700 \par\addvspace{\medskipamount}}
\newenvironment{cor*}{\subsection*{Corollary} \it}{\penalty-700 \par\addvspace{\medskipamount}}

\newcounter{newlistcounter}
\newenvironment{parenenumerate}{\begin{list}%
{{\upshape (\arabic{newlistcounter})}}%
{\usecounter{newlistcounter}%
\setlength{\itemsep}{0em}\setlength{\parsep}{0em}\setlength{\topsep}{0.5em}%
\setlength{\itemindent}{0em}}}{\end{list}}

\newenvironment{shortbulletlist}{\begin{list}%
{$\bullet$}%
{\usecounter{newlistcounter}%
\setlength{\itemsep}{0em}\setlength{\parsep}{0em}\setlength{\topsep}{0.5em}%
\setlength{\itemindent}{0em}}}{\end{list}}

\newcommand{\introskip}{\vspace*{0.8\bigskipamount}}

\newcommand{\tn}[1]{\textnormal{#1}}

\newcommand{\iso}{\cong}

\newcommand{\hmap}{\mathbf{h}}

\newcommand{\PsAlg}[1]{\tn{Ps-}{#1}\tn{-Alg}}

\newcommand{\Algl}[1]{{#1}\tn{-Alg}_{\tn{l}}}

\newcommand{\Alg}[1]{{#1}\tn{-Alg}}
\newcommand{\Algs}[1]{{#1}\tn{-Alg}_{\tn{s}}}

\newcommand{\PSh}[1]{\widehat{#1}}
\newcommand{\catequiv}{\simeq}
\newcommand{\yoneda}{\mathsf{y}}

\newcommand{\kat}[1]{\text{\textbf{\textsl{#1}}}}
\newcommand{\Set}{\kat{Set}}
\newcommand{\Cat}{\kat{Cat}}
\newcommand{\Grpd}{\kat{Grpd}}
\newcommand{\SymMonCat}{\kat{SMC}}
\newcommand{\sSymMonCat}{\kat{SMC}_{\tn{s}}}
\newcommand{\pSymMonCat}{\kat{pSMC}}
\newcommand{\spSymMonCat}{\kat{pSMC}_{\tn{s}}}
\newcommand{\Opd}{\kat{Opd}}
\newcommand{\pOpd}{\kat{pOpd}}
\newcommand{\spOpd}{\kat{pOpd}_{\tn{s}}}
\newcommand{\Subst}{\kat{Subst}}
\newcommand{\RPat}{\kat{RPat}}
\newcommand{\FeynCat}{\kat{FeynCat}}
\newcommand{\sRPat}{\kat{RPat}_{\tn{s}}}
\newcommand{\freesym}{\mathsf{S}}
\newcommand{\freemon}{\mathsf{M}}

\newcommand{\LHerm}{\mathsf{F}}

\newcommand{\pEnd}{\End_{\tn{p}}}
\newcommand{\pLHerm}{\LHerm_{\tn{p}}}
\newcommand{\iopEnd}{\End_{\tn{iop}}}
\newcommand{\iopLHerm}{\LHerm_{\tn{iop}}}
\newcommand{\ca}[1]{\mathcal{#1}}

\newcommand{\SMCMnd}{\freesym}
\newcommand{\id}{\operatorname{id}}

\newcommand{\obj}{\operatorname{obj}}
\newcommand{\End}{\operatorname{End}}

\newcommand{\lan}{\operatorname{lan}}
\newcommand{\isleftadjointto}{\dashv}
\newcommand{\op}{^{\text{{\rm{op}}}}}
\newcommand{\tensor}{\otimes}
\newcommand{\Tensor}{\textstyle{\bigotimes}}
\newcommand{\isopil}{\stackrel{\raisebox{0.1ex}[0ex][0ex]{\(\sim\)}}%
			{\raisebox{-0.15ex}[0.28ex]{\(\rightarrow\)}}}
\newcommand{\comma}{\raisebox{1pt}{$\downarrow$}}

\newcommand{\JJ}{\mathsf{J}}
\newcommand{\GPin}{\mathsf{G}}

\DeclareMathOperator*{\opAst}{\Asterisk}

\newcommand{\rhomark}{\sigma}

\newcommand{\oldlambda}{\gamma}

\newcommand{\sigmarhof}{\sigma_\alpha}

\newcommand{\gammatog}{g}   
\newcommand{\secondobjectinC}{x'}   
\newcommand{\arrowinC}{f}           
\newcommand{\mapofpresheaves}{u} 

\newcommand{\betacell}[1]{\theta^{#1}}
\newcommand{\oldw}{\tau_{\operatorname{comma}}}
\newcommand{\someoperation}{b}

\begin{document}

\setcounter{page}{148}

\title
[Regular patterns, substitudes, Feynman categories, operads]
{Regular patterns, substitudes, \\ Feynman categories and operads}

\copyrightyear{2018}

\author{Michael Batanin, Joachim Kock, and Mark Weber}

\address{Department of Mathematics, Macquarie University\\[5pt]
Departament de matem\`atiques, Universitat Aut\`onoma de Barcelona\\[5pt]
Faculty of Mathematics and Physics, Charles University, Prague\\[5pt]
}

\eaddress{michael.batanin@mq.edu.au
\CR
kock@mat.uab.cat
\CR
mark.weber.math@gmail.com}


\makeatletter

\def\maketitle{\par
 \begingroup
 \def\thefootnote{\fnsymbol{footnote}}
 \def\@makefnmark{\hbox
 to\z@{$^{\@thefnmark}$\hss}}
 \global\@topnum\z@ \@maketitle \thispagestyle{plain}
 \endgroup
 \setcounter{footnote}{0}
 \let\@maketitle\relax
{\def\thempfn{}
\footnotetext{Published in {\em Theory and Applications of 
Categories}, Vol.~33, No.~7, 2018, pp.~148--192.}
\footnotetext{\url{http://tac.mta.ca/tac/volumes/33/7/33-07abs.html}}
\ifthanks\footnotetext{\@thanks}\fi
\ifamsclass\footnotetext{\@amsclass}\fi
\ifkeywords\footnotetext{\@keywords}\fi
\def\relaxp[##1]{\relax}
\def\footnotemark{\@ifnextchar[{\relaxp}{\relax}}
\footnotetext{\copyright\ \@cauthor,\ \@copryear.
Permission to copy for private use granted.}
}}

\makeatother

\maketitle

\begin{abstract}
  We show that the regular patterns of Getzler (2009) form a
  2-category biequivalent to the 2-category of substitudes of Day and
  Street (2003), and that the Feynman categories of Kaufmann and Ward
  (2013) form a 2-category biequivalent to the 2-category of coloured
  operads (with invertible $2$-cells).
  These biequivalences induce equivalences between the
  corresponding categories of algebras.  There are three main
  ingredients in establishing these biequivalences.  The first is a
  strictification theorem (exploiting Power's General Coherence
  Result) which allows to reduce to the case where the structure maps
  are identity-on-objects functors and strict monoidal.  Second, we
  subsume the Getzler and Kaufmann--Ward hereditary axioms into the
  notion of Guitart exactness, a general condition ensuring
  compatibility between certain left Kan extensions and a given monad,
  in this case the free-symmetric-monoidal-category monad.  Finally we
  set up a biadjunction between substitudes and what we call pinned
  symmetric monoidal categories, from which the results follow as a
  consequence of the fact that the hereditary map is precisely the
  counit of this biadjunction.
\end{abstract}

\setcounter{tocdepth}{1}

\setcounter{section}{-1}

\section{Introduction and overview of results}

A proliferation of operad-related structures have seen the light in the past
decades, such as modular and cyclic operads, properads, and props.
Work of many people has sought to develop categorical formalisms covering all
these notions on a common footing, and in particular to describe adjunctions
induced by the passage from one type of structure to another as a
restriction/Kan extension pair \cite{Batanin:0207281, Batanin-Berger:1305.0086,
Batanin-Markl:1404.3886, Borisov-Manin:0609748, Costello:0402015, Day-Street:substitudes,
Fiore-Gambino-Hyland-Winskel:JLMS, Gambino-Joyal:1405.7270,
Getzler:0701767, Markl:0601129, Weber-Guitart}. For the line of development of the present work, the work of Costello~\cite{Costello:0402015} was especially
inspirational: in order to construct the modular envelope of a cyclic operad, he
presented these notions as symmetric monoidal functors out of certain symmetric
monoidal categories of trees and graphs, and arrived at the modular envelope as
a left Kan extension corresponding to the inclusion of one symmetric monoidal
category into the other.  Unfortunately it is not clear from this construction
that the resulting functor is even symmetric monoidal.  The problem was
addressed by Getzler~\cite{Getzler:0701767} by identifying a condition needed
for the construction to work: he introduced the notion of a `regular pattern'
(cf.~\ref{def:rpat-intro} below), which includes a condition formulated in terms
of Day convolution, and which guarantees that constructions like Costello's will work.
However, his condition is not always easy to verify in practice.  Meanwhile,
Markl~\cite{Markl:0601129}, and later Borisov and
Manin~\cite{Borisov-Manin:0609748}, studied general notions of graph categories,
designed with generalised notions of operad in mind, and isolated in particular
a certain hereditary condition, which has also been studied by Melli\`es and
Tabareau~\cite{MelliesTabareau-TAlgTheoriesKan} within a different formalism
(cf.~\ref{hered} below).  This condition found a comma-category formulation in
the recent work of Kaufmann and Ward~\cite{Kaufmann-Ward:1312.1269}, being the
essential axiom in their notion of `Feynman category', cf.~\ref{def:FC-intro}
below.

Kaufmann and Ward notice that the hereditary axiom is closely related to the
Day-convolution Kan extension property of Getzler, and 
provide an easy-to-check
condition under which the envelope construction (and other constructions given
by left Kan extensions) work.  Their work is the starting point for our
investigations.

Another common generalisation of operads and symmetric monoidal 
categories are the substitudes of Day and
Street~\cite{Day-Street:substitudes,Day-Street:lax-monoids} 
(in fact considered briefly already by 
Baez and Dolan~\cite{Baez-Dolan:9702} under the name $C$-operad).
Their interest came from the study of a
nonstandard convolution construction introduced by Bakalov, D'Andrea, and
Kac~\cite{Bakalov-DAndrea-Kac:0007121}. Substitudes can be also
understood as monads in the bicategory of generalised species,
introduced by Fiore, Gambino, Hyland and
Winskel~\cite{Fiore-Gambino-Hyland-Winskel:JLMS} in 2008.

In the present paper we prove that regular patterns are essentially the same
thing as substitudes, and that Feynman categories are essentially the same thing
as (coloured) operads.  More precisely, we establish biequivalences of
$2$-categories---this is the best sameness one can hope for, since the involved
structures are categorical and hence form $2$-categories.  For all four notions,
a key aspect is their algebras.  We show furthermore that under the
biequivalences established, the notions of algebras agree.  More precisely, if a
regular pattern and a substitude correspond to each other under the
biequivalence, then their categories of algebras are equivalent.
Similarly of course with Feynman categories and operads.

In a broader perspective, our results can be seen as part of a dictionary
between two approaches to operad-like structures and their algebras, namely the
symmetric-monoidal-category approach and the operadic/multicategorical approach.
This dictionary goes back to the origins of operad theory, cf.~Chapter 2 of
Boardman--Vogt~\cite{Boardman-Vogt:LNM347}.  In fact to establish the results we
exploit a third approach, namely that of $2$-monads, which goes back to Kelly's
paper on clubs \cite{Kelly:clubs}.  This more abstract approach allows us to
pinpoint some essential mechanisms in both approaches.  In particular we subsume
the Getzler and Kaufmann--Ward hereditary axioms into the $2$-categorical notion
of Guitart exactness, a general condition ensuring compatibility between certain
left Kan extensions and a given $2$-monad, in this case the
free-symmetric-monoidal-category monad. 
Since the notion of Guitart exactness has recently proved very
useful in operad theory and abstract homotopy theory~\cite{
Batanin-Berger:1305.0086,
Groth-Stovicek:1401.6451,
Kahn-Maltsiniotis, 
Maltsiniotis:carres,
Weber-Guitart},
this interpretation of the axioms of Getzler and
Kaufmann--Ward is of independent interest, and we elaborate on it in some 
detail.

The equivalence between regular patterns and substitudes does not seem
to have been foreseen by anybody.  The equivalence between Feynman
categories and operads may come as a surprise, as Kaufmann and Ward in
fact introduced Feynman categories with the intention of providing an
`improvement' over the theory of operads.  Part of the structure of
Feynman category is an explicit groupoid, which one might think of as
a `groupoid of colours', in contrast to the {\em set} of colours of an
operad.  Rather unexpectedly, this groupoid is now revealed to be
available already in the usual notion of operad, namely as the
groupoid of invertible unary operations.  Since the notion of operad
is undoubtedly fundamental, the equivalence we establish attests to
the importance also of the notion of Feynman category, now to be
regarded as a useful alternative viewpoint on operads.

In the present paper, for the sake of focusing on the principal
ideas, we work only over the category of sets.  For the enriched
setting, we refer to Caviglia~\cite{Caviglia:1510.01289}, who 
independently has established an enriched version of the equivalence 
between Feynman categories and operads.

\introskip

We proceed to state our main result, and sketch the ingredients that go into 
its proof.

\introskip

For $C$ a category, we denote by $\freesym C$ the free symmetric monoidal 
category on $C$.

\begin{blanko}{Definition of regular pattern.}\label{def:rpat-intro}
   (Getzler~\cite{Getzler:0701767}) A {\em regular pattern} is a symmetric strong
   monoidal functor $\tau: \freesym C \to M$ such that
  \begin{parenenumerate}
    \item $\tau$ is essentially surjective

    \item the induced functor of presheaves
    $\tau^* : \PSh M \to \PSh {\freesym C}$ is
    strong monoidal for the Day convolution tensor product.
  \end{parenenumerate}
\end{blanko}

\begin{blanko}{Definition of Feynman category.}\label{def:FC-intro}
   (Kaufmann--Ward~\cite{Kaufmann-Ward:1312.1269}) A {\em Feynman category} is a
   symmetric strong monoidal functor $\tau: \freesym C \to M$ such that
  \begin{parenenumerate}
    \item $C$ is a groupoid
    
    \item $\tau$ induces an equivalence of groupoids $\freesym C \isopil 
    M_{\operatorname{iso}}$

    \item $\tau$ induces an equivalence of groupoids
    $
    \freesym(M\comma C)_{\operatorname{iso}} \isopil (M\comma M)_{{\operatorname{iso}}}
    $.
  \end{parenenumerate}
\end{blanko}

\begin{blanko}{The hereditary condition.}
  Getzler's definition is staged in the enriched setting.  Kaufmann
  and Ward also give an enriched version called {\em weak Feynman
  category} (\cite{Kaufmann-Ward:1312.1269}, Definition~4.2 and
  Remark~4.3) which over $\Set$ reads as follows:

  $\tau:\freesym C \to M$ is an essentially surjective symmetric
  strong monoidal functor, and the following important {\em
  hereditary} condition holds (formulated in more detail in
  \ref{hered}): For any $x_1,\ldots,x_m, y_1,\ldots,y_n \in C$, the
  natural map given by tensoring
  $$
  \sum_{\alpha:\underline m \to \underline n} \prod_{j\in \underline n} 
  M( \bigotimes_{i \in \alpha^{-1}(j)} \tau x_i, \tau y_j)
  \longrightarrow
  M(\bigotimes_{i \in \underline m} \tau x_i, \bigotimes_{j\in \underline 
  n} \tau y_j)
  $$
  is a bijection.  (They recognise that this weak notion is `close' to
  Getzler's notion of regular pattern but do not prove that it is
  actually equivalent.  In fact this condition does not really play a
  role in the developments in \cite{Kaufmann-Ward:1312.1269}.)
  
  The hereditary condition is natural from a combinatorial viewpoint
  where it says that every morphism splits into a tensor product of
  `connected' morphisms.  We shall see (\ref{lem:hered=ff}) that in
  the essentially surjective case it is exactly the condition that the
  counit for the substitude Hermida adjunction is fully faithful.
\end{blanko}

\begin{blanko}{Operads and substitudes.}
  By {\em operad} we mean coloured symmetric operad in $\kat{Set}$.  We refer to
  the colours as objects.  The notion of substitude was introduced by Day and
  Street~\cite{Day-Street:substitudes}, as a general framework for substitution
  in the enriched setting.  Our substitudes are their symmetric substitudes,
  cf.~also \cite{Batanin-Berger-Markl:0906.4097}, whose appendix constitutes a
  concise reference for the basic theory of substitudes.  A quick definition is
  this (cf.~\cite[6.3]{Day-Street:lax-monoids}): a {\em substitude} is an operad
  equipped with an identity-on-objects operad morphism from a category (regarded as
  an operad with only unary operations).
\end{blanko}


We can now state the main theorem:

\begin{thm*}
  {\upshape (Cf.~Theorem~\ref{thm:Im(F)} and Theorem~\ref{thm:FC=opd}.)}
  There is a biequivalence between the $2$-category of substitudes and the
  $2$-category of regular patterns.  It restricts to a biequivalence between the
  $2$-category of operads (with invertible $2$-cells)
  and the $2$-category of Feynman categories (with invertible $2$-cells).
\end{thm*}


The biequivalence means that when going back and forth, not an
isomorphic object is obtained, but only an equivalent one.  This is a
question of strictification: one ingredient in the proof is to show
that every regular pattern is equivalent to a strict one, and a
variant of the main theorem can be stated as a $1$-equivalence between
these {\em strict} regular patterns and substitudes.  It should be
observed that equivalent regular patterns have equivalent algebras
(\ref{prop:algebras}).

\introskip

We briefly run through the main ingredients of the proof, and outline
the contents of the paper.

In Section~\ref{sec:strict}, we show that regular patterns and Feynman
categories can be strictified.  Both notions concern a symmetric
strong monoidal functor $\tau:\freesym C \to M$ where $\freesym C$ is
the free symmetric monoidal category on a category $C$, and in
particular is strict.  The main result is this:

\begin{prop*}
  {\upshape (Cf.~Proposition~\ref{prop:strictregpat}.)}
  Every essentially surjective symmetric strong monoidal functor $\freesym C \to M$, is
  equivalent to one $\freesym C \to M'$, for which $M'$ is a symmetric
  strict monoidal category, and $\freesym C \to M'$ is strict monoidal and
  identity-on-objects.
\end{prop*}

\noindent
This is a consequence of Power's coherence result
\cite{Power-GeneralCoherenceResult}, recalled in the appendix.
Since the notions of regular pattern and Feynman
category are invariant under monoidal equivalence, we may as well work
with the strict case, which will facilitate the arguments
greatly, and highlight the essential features of the notions, over the
subtleties of having coherence isomorphisms everywhere.

\introskip

The next step, which makes up Section~\ref{sec:exact},
is to put Getzler's condition (2) into the context of Guitart exactness.

\begin{blanko}{Guitart exactness.}
  Guitart~\cite{Guitart:1980} introduced the notion of exact square: they are those
  squares that pasted on top of a pointwise left Kan extension again gives a
  pointwise left Kan extension.  A morphism of $T$-algebras for a monad $T$ is
  exact when the algebra morphism coherence square is exact.  We shall need this notion
  only in the case where $T$ is the free symmetric monoidal category monad on
  $\Cat$: it thus concerns symmetric monoidal functors.
\end{blanko}

\begin{thm*}
  {\upshape (Cf.~Theorem~\ref{thm:exactness-convolution-characterisations}.)}
  The following are equivalent for a symmetric colax monoidal functor $\tau : S \to M$.

  \begin{parenenumerate}
    \item $\tau$ is Guitart exact
  
    \item left Kan extension of Yoneda along $\tau$ is strong monoidal
    
    \item left Kan extension of any strong monoidal functor along $\tau$ is again strong 
    monoidal
  
    \item $\tau^* : \PSh M \to \PSh S$ is strong monoidal
  
    \item a certain category of factorisations is connected (cf.~Lemmas
    \ref{lem:combinatorial-exactness-Cat-fact-version} and \ref{lem:Fact})
  \end{parenenumerate}
\end{thm*}

In the special case of interest to us we thus have

\begin{cor*}
  For a symmetric strong monoidal functor $\tau:\freesym C \to M$, axiom (2) of being
  a regular pattern is equivalent to being exact.
\end{cor*}

In Section~\ref{sec:Hereditary} we analyse the hereditary condition,
also shown to be equivalent to a special case of Guitart exactness:

\begin{prop*}
  {\upshape (Cf.~Proposition~\ref{prop:rPat=hered}.)}
  An essentially surjective symmetric strong monoidal functor $\tau : \freesym C 
  \to M$ is exact if and only if it satisfies the hereditary condition.  
\end{prop*}

\begin{cor*}
  A regular pattern is a symmetric strong monoidal functor 
  $\tau:\freesym C \to M$ which is essentially surjective and 
  satisfies the hereditary condition.
\end{cor*}

Axiom~(3) of the notion of Feynman category of Kaufmann and
Ward~\cite{Kaufmann-Ward:1312.1269}, the equivalence of comma categories 
$\freesym(M\comma C)_{\operatorname{iso}} \isopil (M\comma 
M)_{{\operatorname{iso}}}$, is of a slightly different flavour to the other
related conditions (and in particular, does not seem to carry over to the
enriched context).  While it is implicit in \cite{Kaufmann-Ward:1312.1269}
that this condition is essentially equivalent to the hereditary condition,
the relationship is actually involved enough to warrant a detailed proof,
which makes up our Section~\ref{sec:comma}.

The outcome is the following result, essentially proved by Kaufmann and 
Ward~\cite{Kaufmann-Ward:1312.1269}.

\begin{cor*}
  A Feynman category is a special case of a regular pattern,
  namely such that $C$ is a groupoid and 
  $\freesym C \to M_{\operatorname{iso}}$ is an equivalence.
\end{cor*}

With these two corollaries in place, we can finally establish the 
promised biequivalences in Section~\ref{sec:pins}.  We achieve this
by setting up {\em pinned} variations of the symmetric Hermida 
adjunction~\cite{Hermida-RepresentableMulticategories} between 
symmetric monoidal categories and operads:

\begin{blanko}{Pinned symmetric monoidal categories and pinned operads.}
  A {\em pinned symmetric monoidal category} is a symmetric monoidal
  category $M$ equipped with a symmetric strong monoidal functor
  $\freesym C \to M$ (where $\freesym C$ is the free symmetric
  monoidal category on some category $C$).  Hence regular patterns and
  Feynman categories are examples of pinned symmetric monoidal
  categories.  Similarly, a {\em pinned operad} is defined to be an
  operad equipped with a functor from a category, viewed as an operad
  with only unary operations.  Substitudes are thus pinned operads for
  which the structure map is identity-on-objects.  The latter
  condition exhibits substitudes as a coreflective subcategory of
  pinned operads.
\end{blanko}

\begin{blanko}{The substitude Hermida adjunction.}
  Our main result will follow readily from the following variation on the 
  Hermida adjunction---actually a biadjunction,
  which goes between pinned symmetric monoidal categories
  and substitudes via pinned operads:
  \begin{equation}\label{eq:intro-adj}
  \xymatrix{
    \pSymMonCat \ar@<-5pt>[r] \ar@{}[r]|-{\scriptstyle{\bot}}
    & \ar@<-5pt>[l] \pOpd \ar@<-5pt>[r] \ar@{}[r]|-{\scriptstyle{\bot}}
    & \ar@<-5pt>[l] \Subst   .
    }
  \end{equation}

  The right adjoint takes a pinned symmetric monoidal
  category $\freesym C \to M$ to the substitude
  \[
  C \to \End(M)|C ,
  \]
  the endomorphism operad on $M$, base-changed to $C$.  It is a important feature
  of substitudes (not enjoyed by operads) that they can be base-changed along
  functors.

  The left adjoint in \eqref{eq:intro-adj} takes a substitude $C \to P$ to the
  pinned symmetric monoidal category $\freesym C \to \LHerm P$, where $\LHerm P$ is the
  free symmetric monoidal category on $P$ as in the ordinary Hermida adjunction:
  the objects of $\LHerm P$ are finite sequences of objects in $P$, and its
  arrows from sequence $x_1,\ldots,x_m$ to sequence $y_1,\ldots,y_n$ are given
  by
  \[
  \LHerm P( \mathbf{x}, \mathbf{y}) :=
      \sum_{\alpha:\underline m \to \underline n} 
      \prod_{j\in \underline n}
      P( (x_i)_{i \in \alpha^{-1}(j)} , y_j) .
  \]
  The left adjoint is now shown to be fully faithful (\ref{prop:unit-iop-iso}).
  This important feature is not shared by the original 
  Hermida adjunction.
  Our key result characterises the image of the left adjoint
  by determining where the counit is invertible:
\end{blanko}

\begin{prop*}
  {\upshape (Cf.~Proposition~\ref{prop:substitudes-among-pinned-SMCs}.)}
  The counit $\varepsilon_\tau$ is an equivalence if and only if $\tau$ is
  essentially surjective and the hereditary condition holds.
\end{prop*}

\begin{cor*}
  The essential image of the left adjoint is the $2$-category of regular
  patterns.
\end{cor*}

In particular, this establishes the first part of the Main Theorem:

\begin{thm*}
  {\upshape (Cf.~Theorem~\ref{thm:Im(F)}.)}
  The left adjoint induces a biequivalence between
  the $2$-category of substitudes and the $2$-category of regular patterns.
\end{thm*}

To an operad $P$ one can assign a substitude by taking the canonical
groupoid pinning $P_1^{\operatorname{iso}} \to P$.  This is not
functorial in all $2$-cells, only in invertible ones; it is the object
part of a fully faithful $2$-functor $(\Opd)^{\operatorname{2-iso}}
\to \Subst$.  We characterise its regular patterns in the image of
this $2$-functor: they are precisely the Feynman categories.  This
establishes the second part of the main theorem:

\begin{thm*}
  {\upshape (Cf.~Theorem~\ref{thm:FC=opd}.)}
  The previous biequivalence induces a biequivalence between the 
  $2$-category of operads (with invertible $2$-cells) and the 
  $2$-category of Feynman categories (with invertible $2$-cells).
\end{thm*}

\noindent
{\sc Acknowledgments.} The authors thank Ezra Getzler, Ralph Kaufmann,
and Richard Garner for fruitful conversations, and thank the anonymous
referee for many pertinent remarks and suggestions that led to some
simplifications.  The bulk of this work was carried out while J.K.\ was
visiting Macquarie University in February--March 2015, sponsored by
Australian Research Council Discovery Grant DP130101969.  M.B.\
acknowledges the financial support of Scott Russell Johnson Memorial
Foundation, J.K.\ was supported by grant number MTM2013-42293-P of
Spain, and M.W.\ by grant number GA CR P201/12/G028 from the Czech
Science Foundation.  Both M.B.\ and M.W.\ acknowledge the support of the
Australian Research Council grant No.~DP130101172.

\section{Strictification of regular patterns}

\label{sec:strict}

\begin{blanko}{The free symmetric monoidal category.}
The free symmetric monoidal category $\freesym C$ on a category $C$ has the following explicit description. The objects of $\freesym C$ are the finite sequences of objects of $C$. A morphism is of the form
\[ (\rho,(f_i)_{i \in \underline n}) : (x_i)_{i \in \underline n} \longrightarrow (y_i)_{i \in \underline n} \]
where $\rho \in \Sigma_n$ is a permutation, and for $i \in \underline{n} =
\{1,...,n\}$, $f_i : x_i \to y_{\rho i}$.  Intuitively such a morphism is a
permutation labelled by arrows of $C$, as in
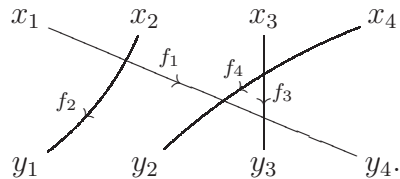
\[ \xygraph{!{0;(1.25,0):(0,1.25)::} {x_1}="t1" [r] {x_2}="t2" [r] {x_3}="t3" [r] {x_4}="t4"
"t1" [d] {y_1}="b1" [r] {y_2}="b2" [r] {y_3}="b3" [r] {y_4.}="b4"
"t1"-"b4"|(.43)@{>}^(.38){f_1} "t2"-@/^{.5pc}/"b1"|(.6)@{>}_(.6){f_2} "t3"-"b3"|(.57)@{>}^(.5){f_3} "t4"-@/_{.5pc}/"b2"|(.57)@{>}_(.54){f_4}} \]
For further details, see the Appendix, where it is explained and exploited that $\freesym$ underlies a $2$-monad on $\Cat$.  It will be important that $\freesym C$ is actually a symmetric {\em strict} monoidal category.
\end{blanko}

\begin{blanko}{Gabriel factorisation.}
  Given a functor $F:C \to D$, its factorisation into an identity-on-objects
  followed by a fully faithful functor is referred to as the {\em Gabriel
  factorisation} of $F$:
  $$\xymatrix{
      C\ar[rr]^F\ar[rd]_{\text{i.o.}} && D \\
     & D' \ar[ru]_{\text{f.f.}} & 
  }$$
\end{blanko}

\begin{prop}\label{prop:strict}
  Let $F:S \to M$ be a symmetric strong monoidal functor, and assume that
  $S$ is a symmetric {\em strict} monoidal category.  Then for the Gabriel factorisation
  $$\xymatrix{
      S\ar[rr]^F\ar[rd]_{G} && M \\
     & M' \ar[ru]_{H} & 
  }$$
  there is a canonical symmetric {\em strict} monoidal structure on $M'$
  for which $G$ is a symmetric strict monoidal functor, and $H$ is canonically symmetric 
  strong monoidal.
\end{prop}
The proof, relegated to the Appendix, exploits the general coherence result
of Power~\cite{Power-GeneralCoherenceResult}. A direct proof is somewhat subtle
because the monoidal structure on $M'$ is constructed as a mix of the monoidal 
structures on $S$ and on $M$, and it is rather cumbersome to check that the
trivial associator defined on $M'$ is actually natural. Instead following Power's 
approach gives an elegant abstract proof, which exploits the following easily
checked facts: (1) the Gabriel factorisation has a $2$-dimensional aspect where
isomorphisms can always be shifted right in the factorisation; and (2) $\freesym$
preserves this factorisation.


\begin{blanko}{Pinned symmetric monoidal categories.}
\label{bl:pSMC}
  The following terminology will be justified in Section~\ref{sec:pins}, as
  part of further pinned notions.  A {\em pinned symmetric monoidal
  category} is a symmetric monoidal category $M$ equipped with a symmetric
  strong monoidal functor $\tau : \freesym C \to M$ (where $C$ is some
  category). Pinned symmetric monoidal categories are the objects of a
  2-category $\pSymMonCat$.  A morphism $(C_1,\tau_1,M_1) \to 
  (C_2,\tau_2,M_2)$ is a triple
  $(F,G,\omega)$ consisting of a functor $F : C_1 \to C_2$, a symmetric strong
  monoidal functor $G : M_1 \to M_2$, and an invertible monoidal natural
  transformation $\omega :\tau_2 \circ \freesym F \simeq G \circ \tau_1$.  A
  $2$-cell $(F,G,\omega) \to (F',G',\omega')$ is a pair
  $(\alpha,\beta)$, where $\alpha :F \to F'$ is a natural transformation
  and $\beta : G \to G'$ is a monoidal natural transformation, such that
  $\omega$ pasted with $\beta$ equals $\freesym \alpha$ pasted with
  $\omega'$.
  
  A pinned symmetric monoidal category $\tau: \freesym C \to M$ is called
  {\em strict} when $M$ is a symmetric strict monoidal category and $\tau$
  is a symmetric strict monoidal functor.  A morphism $(F,G,\omega) :
  (C_1,\tau_1,M_1) \to (C_2,\tau_2,M_2)$ is {\em strict} when $G$ is a symmetric
  strict monoidal functor and $\omega$ is the identity.  The locally full
  sub-$2$-category spanned by the strict objects and strict morphisms is
  denoted $\spSymMonCat$.
\end{blanko}

\begin{blanko}{Regular patterns.}\label{RPat}
  (Getzler~\cite{Getzler:0701767}) A {\em regular pattern} is a symmetric
  strong monoidal functor $\tau: \freesym C \to M$ such that
  \begin{parenenumerate}
  \item $\tau$ is essentially surjective
  \item the induced functor of presheaves
  $\tau^* : \PSh M \to \PSh {\freesym C}$ is
  strong monoidal for the Day convolution tensor product.
  \end{parenenumerate}
  Regular patterns form a $2$-category $\RPat$,
  namely the full sub-$2$-category of $\pSymMonCat$ spanned by the regular patterns.

  A regular pattern (resp.~a morphism of regular patterns) is called {\em strict} when
  it is strict as a pinned symmetric monoidal category (resp.~a morphism 
  of pinned symmetric monoidal categories).  These form thus a full sub-$2$-category
  $\sRPat \subset \spSymMonCat$ and a locally full sub-$2$-category
  $\sRPat \subset \RPat$. 
\end{blanko}

\begin{prop}\label{prop:strictregpat}
  The inclusion $2$-functor $\sRPat \subset \RPat$ is a biequivalence.
\end{prop}

\begin{proof}
  If $\tau: \freesym C \to M$ is a regular pattern, in its Gabriel factorisation
  (as in Proposition~\ref{prop:strict})
  $$\xymatrix{
      \freesym C\ar[rr]^\tau\ar[rd]_{\sigma} && M \\
     & M' \ar[ru]_{\phi} & 
  }$$
  $\sigma$ is identity-on-objects, and $\phi$ is a strong monoidal equivalence
  (say with pseudo-inverse $\psi$ and $2$-cell $\omega:\psi\phi\simeq
  \id_{M'}$).  It follows that $\sigma^*: \widehat M \to \widehat{\freesym C}$
  is strong monoidal: in any case it is lax monoidal, and since both $\tau^*$
  and $\phi^*$ are strong monoidal, also $\sigma^*$ is strong monoidal.  In
  other words, $\sigma$ is a strict regular pattern.  The triple $(\id_C, \phi,
  \id) : \sigma \to \tau$ is an equivalence in
  $\RPat$ with pseudo-inverse $(\id_C, \psi,
  \omega\sigma)$.  This shows that $\tau$ is equivalent to a strict regular
  pattern, so the inclusion $2$-functor is essentially surjective on objects.
  
  The inclusion $2$-functor is locally fully faithful by construction, so
  it remains to see it is locally essentially surjective, i.e.~essentially 
  surjective on morphisms. We need to show, given strict regular patterns
  $\sigma$ and $\sigma'$,
  that any 
  morphism $(F,G,\omega):\sigma \to \sigma'$ is equivalent to a strict one 
  $(F,G_{\text{strict}}, \omega_{\text{strict}}{=}\id)$.
  Since $\sigma$ is bijective on objects, there is a unique way to define the
  strict $G_{\text{strict}}$ on objects so that $\omega_{\text{strict}}$
  becomes the identity $2$-cell.
  It will be equivalent to $G$ by means of the old $\omega$, which also ensures
  the functoriality of $G_{\text{strict}}$.  It is symmetric strict monoidal
  since $\freesym F$ and $\sigma'$ are.  So the inclusion $2$-functor is 
  locally essentially surjective, and hence altogether a biequivalence.
\end{proof}


\begin{blanko}{Algebras.}
  Let $W$ be a symmetric monoidal category.  An {\em algebra} for a
  regular pattern $\tau: \freesym C \to M$ in $W$, is a symmetric
  strong monoidal functor $M \to W$.  With morphisms of algebras given
  by monoidal natural transformations, there is a category
  $\tn{Alg}_{\tau}(W)$ of algebras of $(C,\tau,M)$ in $W$.  A morphism
  of regular patterns $(C,\tau,M) \to (C',\tau',M')$ induces a functor
  $\tn{Alg}_{\tau'}(W) \to \tn{Alg}_{\tau}(W)$ by precomposition.  The
  following proposition is now clear.
\end{blanko}
\begin{prop}
Equivalent regular patterns have equivalent categories of algebras. \qed
\end{prop}
Together with Proposition~\ref{prop:strictregpat}, 
this justifies 
emphasising strict regular patterns, as we shall often
do. This facilitates extracting equivalent characterisations of 
condition (2)---Guitart exactness and the hereditary 
condition, in turn exploited in the final comparison with substitudes.

\section{Guitart exactness}
\label{sec:exact}

In this section and the next we show how the main axioms in the definitions
of regular pattern and Feynman category can be subsumed in the theory of
Guitart exactness.  An important aspect of Guitart exactness is to serve as
a criterion for pointwise left Kan extensions to be compatible with algebraic 
structures.  This direction of the theory is developed rather systematically in
\cite{Weber-Guitart} in the abstract setting of a $2$-monad
on a $2$-category with comma objects.  The interesting case for
the present purposes is the case of the free-symmetric-monoidal-category
monad on $\Cat$, and the issue is then under what circumstances
left Kan extensions are symmetric monoidal functors.

We write $\PSh A := [A\op,\Set]$ for the category of presheaves, and $\yoneda_A:
A \to \PSh A$ for the Yoneda embedding.

\begin{blanko}{Exact squares.}
  A $2$-cell in $\Cat$ of the form
\begin{equation}\label{eq:lax-square-in-Cat}
  \begin{gathered}
\xygraph{!{0;(1.5,0):(0,.6667)::} {P}="p0" [r] {B}="p1" [d] {C}="p2" [l] {A}="p3" "p0":"p1"^-{q}:"p2"^-{g}:@{<-}"p3"^-{f}:@{<-}"p0"^-{p} "p0" [d(.55)r(.4)] :@{=>}[r(.2)]^-{\phi}}
  \end{gathered}
\end{equation}
(called a \emph{lax square}) is {\em exact in the sense of Guitart} \cite{Guitart:1980} when for any natural transformation $\psi$ which exhibits $l$ as a pointwise left Kan extension of $h$ along $f$, the composite
  \[ 
  \xygraph{{P}="tl" [r(2)] {B}="tr" [l(2)d] {A}="l" [r(2)] {C}="r" [dl] 
  {V}="b" "l":"r"^-{f}:"b"^-{l}:@{<-}"l"^-{h} [d(.5)r(.85)] :@{=>}[r(.3)]^-{\psi}
  "tl"(:"tr"^-{q}:"r"^-{g},:"l"_-{p}) "tl" [d(.5)r(.85)] :@{=>}[r(.3)]^-{\phi}} 
  \]
  exhibits $lg$ as a pointwise left Kan extension of $hp$ along $q$.

Suppose that in this situation $A$ is locally small and $f$ is admissible in the
sense \cite{StreetWalters-YonedaStructures, Weber-2Toposes} that $C(fa,c)$ is small
for all $a \in A$ and $c \in C$. One thus has the functor $C(f,1) : C \to
\PSh A$ given on objects by $c \mapsto C(f(-),c)$, and the effect on arrows of
the functor $f$ can be organised into a natural transformation
\[
\xygraph{{A}="p0" [r(2)] {C}="p1" [dl] {\PSh A}="p2"
"p0":"p1"^-{f}:"p2"^-{C(f,1)}:@{<-}"p0"^-{\yoneda_A} "p0" [d(.5)r(.85)] 
:@{=>}[r(.3)]^-{\chi^f}} 
\]
which exhibits $C(f,1)$ as a pointwise left
Kan extension of $\yoneda_A$ along $f$ (see e.g.~\cite{Weber-2Toposes}
Example~3.3).

\end{blanko}

\begin{lemma}\label{lem:exact-yoneda-formulation}
  {\upshape (Cf.~Guitart~\cite{Guitart:1980}.)}
  A lax square \eqref{eq:lax-square-in-Cat} in which $A$ and $P$ are small and
  $f$ is admissible is exact if and only if the composite
  \begin{equation}\label{eq:ple-for-exactness}
    \begin{gathered}
  \xygraph{
  {P}="tl" [r(2)] {B}="tr" [l(2)d] {A}="l" [r(2)] {C}="r" [dl] {\PSh A}="b" "l":"r"^-{f}:"b"^-{C(f,1)}:@{<-}"l"^-{\yoneda_A} [d(.5)r(.85)] :@{=>}[r(.3)]^-{\chi^f}
  "tl"(:"tr"^-{q}:"r"^-{g},:"l"_-{p}) "tl" [d(.5)r(.85)] :@{=>}[r(.3)]^-{\phi}
  }\end{gathered}
  \end{equation}
  exhibits $C(f,1)\circ g$ as a pointwise left Kan extension of $\yoneda_A \circ p$
  along $q$.
  \qed
\end{lemma}
When $f = \yoneda_A$, the $2$-cell $\chi^f$ is the identity, and we get the following.
\begin{cor}\label{cor:exact-ple-squares}
  If $P$ and $A$ are small and
  \[
  \xygraph{!{0;(1.5,0):(0,.6667)::}
  {P}="p0" [r] {B}="p1" [d] {\PSh A}="p2" [l] {A}="p3" 
  "p0":"p1"^-{q}:"p2"^-{g}:@{<-}"p3"^-{\yoneda_A}:@{<-}"p0"^-{p} 
  "p0" [d(.55)r(.4)] :@{=>}[r(.2)]^-{\phi}
  }
  \]
  exhibits $g$ as a pointwise left Kan extension of $\yoneda_A\circ p$ along $q$, then
  $\phi$ is exact.
  \qed
\end{cor}
Exact squares can be recognised in elementary terms in the following way.
First given $a \in A$, $b \in B$ and $\gamma : fa \to gb$ we denote by
$\tn{Fact}_{\phi}(a,\gamma,b)$ the following category.  Its objects are
triples $(\alpha,x,\beta)$ where $x \in P$, $\alpha : a \to px$ and $\beta
: qx \to b$, such that $g(\beta)\phi_xf(\alpha) = \gamma$.  Informally,
such an object is a `factorisation of $\gamma$ through $\phi$'.  A morphism
$(\alpha_1,x_1,\beta_1) \to (\alpha_2,x_2,\beta_2)$ of such is an arrow
$\delta : x_1 \to x_2$ such that $p(\delta)\alpha_1 = \alpha_2$ and
$\beta_1 = \beta_2q(\delta)$.  Identities and compositions are inherited
from $P$.
\begin{lemma}\label{lem:combinatorial-exactness-Cat-fact-version}
  {\upshape \cite{Guitart:1980}}
  A lax square (\ref{eq:lax-square-in-Cat}) in $\Cat$ is exact if 
  and only if for all $a \in A$, $b \in B$, and $\gamma : fa \to gb$, 
  the category $\tn{Fact}_{\phi}(a,\gamma,b)$ defined above is connected.
  \qed
\end{lemma}

\begin{blanko}{Exact monoidal functors.}
In the usual nullary-binary way of writing tensor products in monoidal categories, a \emph{symmetric colax monoidal functor} $f : A \to B$ has coherence morphisms of the form
\[ \begin{array}{lccr} {\overline{f}_0 : fI \longrightarrow I} &&& {\overline{f}_{X,Y} : f(X \tensor Y) \longrightarrow fX \tensor fY} \end{array} \]
(in which $I$ denotes the unit of either $A$ or $B$) which are required to satisfy axioms that express compatibility with the coherences which define the symmetric monoidal structures on $A$ and $B$. Equivalently one can regard a symmetric colax monoidal structure on $f$ as comprising coherence morphisms
\[ \overline{f}_{X_1,...,X_n} : f(X_1 \tensor \cdots \tensor X_n) 
\longrightarrow f(X_1) \tensor \cdots \tensor f(X_n) \]
for each sequence $(X_1,...,X_n)$ of objects of $A$, whose naturality is expressed by the fact that they are the components of a natural transformation 
\[ \xygraph{!{0;(2,0):(0,.5)::}
{\freesym A}="p0" [r] {\freesym B}="p1" [d] {B.}="p2" [l] {A}="p3"
"p0":"p1"^-{\freesym(f)}:"p2"^-{\bigotimes}:@{<-}"p3"^-{f}:@{<-}"p0"^-{\bigotimes}
"p0" [d(.55)r(.4)] :@{=>}[r(.2)]^-{\overline{f}}} \]
We say that $f$ is \emph{exact} when this square is an exact square in the sense discussed above. In terms of the 2-monad $\freesym$, $(f,\overline{f})$ is a colax morphism of pseudo algebras, and in  \cite{Weber-Guitart}, the theory of exact colax morphisms of algebras is developed at the general level of a 2-monad on a 2-category with comma objects.
\end{blanko}

The following lemma is key to the interest in exactness in the present context.
The result is a special case of Theorem~2.4.4 of \cite{Weber-Guitart}. A similar result is obtained in the context of proarrow equipments in \cite{MelliesTabareau-TAlgTheoriesKan} and in a double categorical setting in \cite{Koudenburg-AlgKan}.

\begin{lemma}\label{lem:alg-Kan} {\upshape \cite{Weber-Guitart}}
  Let $f: A \to B$ be an exact symmetric colax monoidal functor.
  Then for any lax symmetric monoidal functor $g: A \to C$ (with $C$ assumed 
  algebraically cocomplete),
  the pointwise left Kan extension $\lan_f g$
  $$\xymatrix{
     A \ar[rr]^f \ar[rd]_g  \ar@{}[rrd]|{\displaystyle{\Rightarrow}}&& B  \ar@{..>}[ld]^{\lan_f g}\\
     & C &
  }$$
  is again naturally lax symmetric monoidal.  Furthermore, if $g$ is strong,
  then so is $\lan_f g$.
  \qed
\end{lemma}

The condition that $C$ is algebraically cocomplete (with respect to $f$) means first
of all that it has enough colimits for the left Kan extension in question to 
exist, and second, that these colimits are preserved by the tensor product in 
each variable.  More formally, 
whenever $\psi$ exhibits $h$ as
a pointwise left Kan extension of $g$ along $f$ as on the left, then the
composite on the right
\[ \xygraph{{\xybox{\xygraph{{A}="l" [r(2)] {B}="r" [dl] {C}="b" 
"l":"r"^-{f}:"b"^-{h}:@{<-}"l"^-{g} [d(.5)r(.85)] :@{=>}[r(.3)]^-{\psi}}}} [r(5)]
{\xybox{\xygraph{{\freesym A}="l" [r(2)] {\freesym B}="r" [dl] {\freesym C}="b" [d] 
{C}="bb" "l":"r"^-{\freesym f}:"b"^-{\freesym h}:@{<-}"l"^-{\freesym g} 
"b":"bb"^{\bigotimes} "l" 
[d(.5)r(.85)] :@{=>}[r(.3)]^-{\freesym\psi}}}}} \]
exhibits $\bigotimes \circ\freesym h$ as a pointwise left Kan extension of 
$\bigotimes \circ\freesym g$ 
along $\freesym f$.

\begin{blanko}{Day convolution tensor product.}
It is well known that the symmetric monoidal structure on $A$ 
(assumed to be small)
extends
essentially uniquely to one on $\PSh A$, for which the tensor product is
cocontinuous in each variable.  This tensor product $\Asterisk : \freesym \PSh
A \to \PSh A$ is called \emph{Day convolution} \cite{Day-Convolution}. It is
folklore that the Day convolution tensor product can also be characterised as a
pointwise left Kan extension as in the following result, which is nothing more
than a translation of the universal property of convolution as expressed by
Im and Kelly \cite{Im-Kelly}, in these terms.
\end{blanko}

\begin{prop}\label{prop:Day-convolution}
{\upshape \cite{Im-Kelly}}
  For $A$ a small symmetric monoidal category, the Day convolution tensor product
  $\Asterisk$ on $\PSh A$ can be
  characterised as the pointwise left Kan extension of $\yoneda_A\circ
  \bigotimes$ along $\freesym \yoneda_A$,
    \[
  \xygraph{!{0;(2,0):(0,.5)::} 
  {\freesym A}="p0" [r] {\freesym \PSh A}="p1"
  [d] {\PSh A .}="p2" [l] A="p3" 
  "p0":"p1"^-{\freesym 
  \yoneda_A}:"p2"^-{\Asterisk}:@{<-}"p3"^-{\yoneda_A}:@{<-}"p0"^-{\bigotimes} 
  "p0" [d(.55)r(.4)] :@{=>}[r(.2)]^-{\overline{\yoneda}_A}}
  \]
  Furthermore, this square (invertible since $\freesym \yoneda_A$ is fully 
  faithful) constitutes the coherence data making $\yoneda_A$ a symmetric strong monoidal
  functor.  Finally, the following universal property holds (usually taken as the
  defining property of the Day convolution tensor product):
  For any cocomplete symmetric monoidal category $X$, composition with 
  $\yoneda_A$ gives equivalences of categories
  \begin{equation*}
  \begin{array}{lccr} {\tn{Cocts}\kat{SMC}_c(\PSh A,X) \catequiv \kat{SMC}_c (A,X)} 
    &&& {\tn{Cocts}\kat{SMC}(\PSh A,X) \catequiv \kat{SMC}(A,X).} \end{array}
  \end{equation*}
\end{prop}
Here $\kat{SMC}$ is the $2$-category of symmetric monoidal categories with symmetric
strong monoidal functors, while $\kat{SMC}_c$ has also symmetric colax monoidal
functors.  The prefixes $\tn{Cocts}$ indicate the full subcategories
spanned by cocomplete symmetric monoidal categories whose tensor product
preserves colimits in both variables.
\begin{proof}
For any category $C$, we denote by $\freemon C$ the free (strict) monoidal
category on $C$.  Explicitly $\freemon C$ is the subcategory of $\freesym C$
containing all the objects, but just the morphisms whose underlying permutation
is an identity.  The inclusions $i_C : \freemon C \to \freesym C$ are the
components of a 2-natural transformation $i : \freemon \to \freesym$ which by
the results of \cite{Weber-PolynomialFunctors}, conforms to the hypotheses of
Proposition~4.6.2 of \cite{Weber-Guitart}.  Thus for any functor $f :C \to D$,
the corresponding naturality square of $i$ on the left
\[ \xygraph{!{0;(2,0):(0,.5)::} 
{\freemon C}="p0" [r] {\freemon D}="p1" [d] {\freesym D}="p2" [l] {\freesym C}="p3"
"p0":"p1"^-{\freemon f}:"p2"^-{i_D}:@{<-}"p3"^-{\freesym f}:@{<-}"p0"^-{i_C}
"p1" [r(1.5)d(.5)]
{\freesym A}="q0" [r] {\freesym \PSh A}="q1" [d] {\PSh A}="q2" [l] A="q3" 
"q0" [u] {\freemon A}="q4" [r] {\freemon \PSh A}="q5"
"q0":"q1"^-{\freesym\yoneda_A}:"q2"^-{\Asterisk}:@{<-}"q3"^-{\yoneda_A}:@{<-}"q0"^-{\bigotimes} "q4" (:"q0"_-{i_A}, :"q5"^-{\freemon\yoneda_A}:"q1"^-{i_{\PSh A}})
"q0" [d(.55)r(.4)] :@{=>}[r(.2)]^-{\overline{\yoneda}_A}} \]
is exact, and so the composite square on the right exhibits $\Asterisk \circ i_{\PSh A}$ as a pointwise left Kan extension, and this functor has the same object map as $\Asterisk$. Computing the left Kan extension on the left in the previous display as a coend in the usual way, one recovers the usual formula for the Day tensor product. Thus the result follows from \cite{Im-Kelly}.
\end{proof}

With Corollary~\ref{cor:exact-ple-squares}, we arrive at the following.
\begin{cor}\label{cor:y-exact}
  $\yoneda_A : (A,\tensor) \to (\PSh A,\Asterisk)$ is exact.
  \qed
\end{cor}

\begin{blanko}{Generalities.}
  For any functor $f: A \to B$ between small categories, we have the
  $2$-cells
  $$
  \xymatrix @=1.4pc  {
     A \ar[rr]^f\ar[rrdd]_{\yoneda_A} \ar@{}[rrrd]|{\overset{\chi^f}\Rightarrow}&& B 
     \ar[rr]^{\yoneda_B} \ar[dd]|{B(f,1)} && \PSh B  \ar@{}[llld]|{\overset{\iota^f}\Rightarrow}
     \ar[lldd]^{f^*} \\
     &&&&\\
     && \PSh A && 
  }
  \qquad\qquad
  \xymatrix{
     A \ar[r]^{\yoneda_A}\ar[d]_f   \ar@{}[rd]|{\overset{l^f}\Rightarrow}
     & \PSh A \ar[d]^{f_!} \\
     B \ar[r]_{\yoneda_B} & \PSh B
  }
  $$
  exhibiting $B(f,1)$ as the pointwise left Kan extension of $\yoneda_A$ along $f$, 
  and $f^*$ (restriction along $f$) as the pointwise left Kan extension of 
  $B(f,1)$ along $\yoneda_B$.    Finally, $l^f$ exhibits $f_!$ (the left adjoint to 
  $f^*$) as
  the pointwise left Kan extension of $\yoneda_B\circ f$ along
  $\yoneda_A$.
Note that $l^f$ is an exact square by Corollary~\ref{cor:exact-ple-squares},
and that both $\iota^f$ and $l^f$ are invertible, since $\yoneda_B$ and $\yoneda_A$ are fully faithful.
\end{blanko}

Returning to our situation of a symmetric colax monoidal functor between small
symmetric monoidal categories $(f,\overline{f}) : A \to B$, by
Lemma~\ref{lem:alg-Kan} $f_!$ gets a symmetric colax monoidal structure from
that of $\yoneda_B\circ f$, since $\yoneda_A$ is exact by Proposition~\ref{prop:Day-convolution}.
The colax coherence datum $\overline{f}_!$ (which we don't make explicit here, and which is invertible if and only if $\overline{f}$ is) induces, by taking mates via $f_!
\isleftadjointto f^*$, the coherence 2-cell $\overline{f}^*$ making $f^*$ a lax
monoidal functor.
Moreover $B(f,1)$ gets a unique monoidal structure
making $\iota^f$ an invertible monoidal natural transformation.
In the context just described we have the following alternative
characterisations of exactness of the symmetric colax monoidal functor
$(f,\overline{f})$.

\begin{thm}\label{thm:exactness-convolution-characterisations} 
  The following statements are equivalent for a colax symmetric monoidal functor
  $f : A \to B$ (assuming $A$ small and $f$ admissible).
  \begin{parenenumerate}
  \item $f$ is exact.
  \label{thmcase:exact}
  \item For any algebraically cocomplete symmetric monoidal category $X$ and any 
	symmetric strong monoidal functor
        $g : A \to X$, the pointwise left Kan extension of $g$ along $f$ 
	is symmetric strong monoidal.
  \label{thmcase:left-extend-along-f}
  \item $B(f,1) : B \to \PSh A$ is symmetric strong monoidal.
  \label{thmcase:Bf1-strong}
  \item $f^* : \PSh B \to \PSh A$ is symmetric strong monoidal.
  \label{thmcase:f*-strong}
  \end{parenenumerate}
\end{thm}

\begin{proof}
(\ref{thmcase:exact}) $\implies$ (\ref{thmcase:left-extend-along-f}): 
The assumptions imply that the left Kan extension exists.  The 
statement now follows from Lemma~\ref{lem:alg-Kan}.

(\ref{thmcase:left-extend-along-f}) $\implies$ (\ref{thmcase:Bf1-strong}): 
By Proposition~\ref{prop:Day-convolution}, $\yoneda_A$ is strong monoidal.
Since $\chi^f$ exhibits $B(f,1)$ as a pointwise left Kan extension of $\yoneda_A$
along $f$, we conclude by the assumption~\eqref{thmcase:left-extend-along-f} that $B(f,1)$ is 
strong monoidal.

(\ref{thmcase:Bf1-strong}) $\implies$ (\ref{thmcase:f*-strong}): 
By Corollary~\ref{cor:y-exact}, $\yoneda_B$ is exact, and $B(f,1)$ is 
strong monoidal by assumption.  But  $\iota^f$ exhibits $f^*$ as the pointwise left Kan extension
of $B(f,1)$ along $\yoneda_B$, so again by Lemma~\ref{lem:alg-Kan} we 
conclude that $f^*$ is strong monoidal.

(\ref{thmcase:f*-strong}) $\implies$ (\ref{thmcase:exact}): From
\cite{StreetWalters-YonedaStructures, Weber-2Toposes} the unit $u$ of the
adjunction $f_!  \isleftadjointto f^*$ is the unique 2-cell satisfying the
equation on the left

\vspace*{-24pt}

\[ \xygraph{!{0;(1.5,0):(0,.75)::} 
A="p0" [r] {\PSh A}="p1" [d] {\PSh B}="p2" [l] {\PSh A}="p3"
"p0":"p1"^-{\yoneda_A}:"p2"^-{f_!}:"p3"^-{f^*}:@{<-}"p0"^-{\yoneda_A}
"p0":"p2"|(.45){\yoneda_Bf}
"p3" [u(.15)r(.2)] :@{=>}[r(.15)]^-{\chi^{\yoneda_Bf}}
"p1" [d(.35)l(.35)] :@{=>}[r(.15)]^-{l^f}
"p1" [r]
A="q0" [r] {\PSh A}="q1" [d] {\PSh B}="q2" [l] {\PSh A}="q3"
"q0":"q1"^-{\yoneda_A}:"q2"^-{f_!}:"q3"^-{f^*}:@{<-}"q0"^-{\yoneda_A}
"q1":"q3"|-{1_{\PSh A}}
"q0" [d(.35)r(.2)] :@{=>}[r(.15)]^-{\id}
"q2" [u(.35)l(.3)] :@{=>}[r(.15)]^-{u}
"p1" [d(.5)r(.5)] {=}
"q1" [r(1.5)u(.5)]
{\SMCMnd \PSh A}="r0" [r] {\SMCMnd \PSh B}="r1" [d] {\PSh B}="r2" [d] {\PSh A}="r3" [ul] {\PSh A}="r4"
"r0":"r1"^-{\SMCMnd f_!}:"r2"^-{\Asterisk}:"r3"^-{f^*}:@/^{1pc}/@{<-}"r4"^-{1_{\PSh A}}:@{<-}"r0"^-{\Asterisk} "r4":"r2"^-{f_!}
"r0" [d(.5)r(.4)] :@{=>}[r(.2)]^-{\overline{f}_!}
"r4" [d(.4)r(.4)] :@{=>}[r(.2)]^-{u}
"r1" [r(1.2)]
{\SMCMnd \PSh A}="s0" [dr] {\SMCMnd \PSh B}="s1" [d] {\PSh B}="s2" [l] {\PSh A}="s3" [u] {\SMCMnd \PSh A}="s4"
"s0":@/^{1pc}/"s1"^-{\SMCMnd f_!}:"s2"^-{\Asterisk}:"s3"^-{f^*}:@{<-}"s4"^-{\Asterisk}:@{<-}"s0"^-{1_{\SMCMnd \PSh A}} "s1":"s4"^-{\SMCMnd f^*}
"s0" [d(.6)r(.4)] :@{=>}[r(.2)]^-{\SMCMnd u}
"s4" [d(.7)r(.4)] :@{=>}[r(.2)]^-{\overline{f}{}^*}
"r1" [d(1)r(.6)] {=}} \]
and $\overline{f}_!$ and $\overline{f}{}^*$ determine each other uniquely by the
equation on the right.  Being the unit of an adjunction, $\SMCMnd u$ is an
absolute pointwise left Kan extension, and since $\overline f{}^*$ is assumed to 
be invertible (\ref{thmcase:f*-strong}), the
common composite of the equation on the right in the previous display exhibits
$f^* \circ \Asterisk_B$ as the pointwise left Kan extension of
$\Asterisk_A$ along $\SMCMnd f_!$. 
Now paste on the left with $\overline{\yoneda}_A$ which is a pointwise left 
Kan extension by
Proposition~\ref{prop:Day-convolution}.
The resulting pointwise left Kan extension can be rewritten as follows:
\[ \xygraph{!{0;(1.25,0):(0,1)::}
{\SMCMnd A}="p0" [ur] {\SMCMnd \PSh A}="p1" [dr] {\SMCMnd \PSh B}="p2" [d] {\PSh B}="p3" [dl] {\PSh A}="p4" [ul] A="p5" [u(.75)r] {\PSh A}="p6"
"p0":"p1"^-{\SMCMnd \yoneda_A}:"p2"^-{\SMCMnd f_!}:"p3"^-{\Asterisk}:"p4"^-{f^*}:@{<-}"p5"^-{\yoneda_A}:@{<-}"p0"^-{\bigotimes}
"p6" (:@{<-}"p1"^-{\Asterisk},:"p3"^-{f_!},:"p4"|-{1_{\PSh A}},:@{<-}"p5"_-{\yoneda_A})
"p0" [r(.45)d(.1)] :@{=>}[r(.1)u(.2)]_-{\overline{\yoneda}_A}
"p2" [l(.55)d(.1)] :@{=>}[r(.1)u(.2)]^-{\overline{f}_!}
"p5" [r(.6)d(.1)] :@{=>}[r(.15)u(.15)]^-{\id}
"p3" [l(.7)] :@{=>}[r(.2)]^-{u}
"p2" [r(1.75)]
{\SMCMnd A}="q0" [ur] {\SMCMnd \PSh A}="q1" [dr] {\SMCMnd \PSh B}="q2" [d] {\PSh B}="q3" [dl] {\PSh A}="q4" [ul] A="q5"
"q1" [d] {\PSh A}="q6" [d] B="q7"
"q0":"q1"^-{\SMCMnd \yoneda_A}:"q2"^-{\SMCMnd f_!}:"q3"^-{\Asterisk}:"q4"^-{f^*}:@{<-}"q5"^-{\yoneda_A}:@{<-}"q0"^-{\bigotimes}
"q6" (:@{<-}"q1"^-{\Asterisk},:"q3"^-{f_!},:@{<-}"q5"_-{\yoneda_A})
"q7" (:"q3"^-{\yoneda_B},:@{<-}"q5"_-{f})
"q0" [r(.45)d(.05)] :@{=>}[r(.1)u(.2)]_-{\overline{\yoneda}_A}
"q2" [l(.55)d(.1)] :@{=>}[r(.1)u(.2)]^-{\overline{f}_!}
"q7" [u(.35)l(.05)] :@{=>}[r(.1)u(.2)]_-{l^f}
"q7" [d(.5)l(.125)] :@{=>}[r(.25)]^-{\chi^{\yoneda_B f}}
"q2" [r(1.75)]
{\SMCMnd A}="r0" [ur] {\SMCMnd \PSh A}="r1" [dr] {\SMCMnd \PSh B}="r2" [d] {\PSh B}="r3" [dl] {\PSh A}="r4" [ul] A="r5"
"r1" [d] {\SMCMnd B}="r6" [d] B="r7"
"r0":"r1"^-{\SMCMnd \yoneda_A}:"r2"^-{\SMCMnd f_!}:"r3"^-{\Asterisk}:"r4"^-{f^*}:@{<-}"r5"^-{\yoneda_A}:@{<-}"r0"^-{\bigotimes}
"r6" (:"r2"^-{\SMCMnd \yoneda_B},:"r7"_-{\bigotimes},:@{<-}"r0"_-{\SMCMnd f})
"r7" (:"r3"^-{\yoneda_B},:"r4"|(.51){B(f,1)},:@{<-}"r5"_-{f})
"r6" [u(.4)l(.15)] :@{=>}[r(.1)u(.2)]_-{\SMCMnd l^f}
"r0" [d(.5)r(.3)] :@{=>}[r(.2)]^-{\overline{f}}
"r6" [d(.5)r(.4)] :@{=>}[r(.2)]^-{\overline{\yoneda}_B}
"r5" [d(.35)r(.5)] :@{=>}[r(.2)]^-{\chi^f}
"r7" [d(.35)r(.25)] :@{=>}[r(.2)]^-{\iota^f}
"p2" [d(.5)r(.875)] {=} "q2" [d(.5)r(.875)] {=}}
\]
and
since $\SMCMnd l^f$ is invertible, we conclude that already 
\[ 
\xygraph{!{0;(1.5,0):(0,.8)::}
{\SMCMnd A}="r0" [r(2)] {\SMCMnd \PSh B}="r2" [d] {\PSh B}="r3" [dl] {\PSh A}="r4" [ul] A="r5"
"r0" [r] {\SMCMnd B}="r6" [d] B="r7"
"r2":"r3"^-{\Asterisk}:"r4"^-{f^*}:@{<-}"r5"^-{\yoneda_A}:@{<-}"r0"^-{\bigotimes}
"r6" (:"r2"^-{\SMCMnd \yoneda_B},:"r7"_-{\bigotimes},:@{<-}"r0"_-{\SMCMnd f})
"r7" (:"r3"^-{y_B},:"r4"|(.51){B(f,1)},:@{<-}"r5"_-{f})
"r0" [d(.5)r(.3)] :@{=>}[r(.2)]^-{\overline{f}}
"r6" [d(.5)r(.4)] :@{=>}[r(.2)]^-{\overline{\yoneda}_B}
"r5" [d(.35)r(.5)] :@{=>}[r(.2)]^-{\chi^f}
"r7" [d(.35)r(.25)] :@{=>}[r(.2)]^-{\iota^f}
}  
\]
exhibits $f^* \circ \Asterisk_B$ as a pointwise left Kan extension of
$\yoneda_A \circ \bigotimes_A$ along $\SMCMnd \yoneda_B \circ \freesym f$.  But
already $\iota^f$ is a pointwise left Kan extension, and $\overline{\yoneda}_B$
is an exact square, so the whole right-hand part of the diagram is a pointwise left Kan
extension.  Since furthermore $\SMCMnd \yoneda_B$ is fully faithful, we can
cancel that right-hand part away (e.g.~by \cite{Kelly:basic-enriched}, Theorem
4.47), so in conclusion also
\[
\xygraph{!{0;(1.5,0):(0,.8)::}
{\SMCMnd A}="s0" [r] {\SMCMnd B}="s6" [d] B="s7" [l(.5)d] {\PSh A}="s4" [l(.5)u] A="s5"
"s4":@{<-}"s5"^-{\yoneda_A}:@{<-}"s0"^-{\bigotimes}
"s6" (:"s7"^-{\bigotimes},:@{<-}"s0"_-{\SMCMnd f})
"s7" (:"s4"^{B(f,1)},:@{<-}"s5"_-{f})
"s0" [d(.5)r(.4)] :@{=>}[r(.2)]^-{\overline{f}}
"s5" [d(.4)r(.4)] :@{=>}[r(.2)]^-{\chi^f}}  ,
\]
 is a pointwise 
left Kan extension.  It now follows from
Lemma~\ref{lem:exact-yoneda-formulation} that $f$ is exact.

Note that the implication (\ref{thmcase:f*-strong}) $\implies$
(\ref{thmcase:left-extend-along-f}) was established already by
Getzler~\cite{Getzler:0701767}, and in fact can be extracted from 
Bunge--Funk~\cite{Bunge-Funk:1999}, Proposition~1.5, as pointed out by the 
anonymous referee.
\end{proof}

\begin{cor}
  For a symmetric monoidal functor $\tau:\freesym C \to M$, axiom (2) of being
  a regular pattern is equivalent to being exact.
  \qed
\end{cor}

\begin{blanko}{Morphisms of regular patterns.}
  Recall (from \ref{RPat}) that a morphism of regular patterns is a 
  diagram of symmetric strong monoidal functors
  $$\xymatrix @R=10pt {
     \freesym C_1 \ar[r]^{\tau_1}\ar[dd]_{\freesym f} 
	 & M_1 \ar[dd]^g \\
	 \ar@<5pt>@{}[r]|{\omega}\ar@<-3pt>@{}[r]|{\simeq}&
	 \\
     \freesym C_2 \ar[r]_{\tau_2} & M_2 ,
  }$$
  where $\omega$ is an invertible monoidal natural transformation.
\end{blanko}

\begin{prop}
  Every such $g:M_1\to M_2$ is exact.
\end{prop}

\begin{proof}
  The free functor $\freesym C_1 \to \freesym C_2$ is exact by Corollary
  4.6.6 of \cite{Weber-Guitart}.  The two functors $\tau_1$ and $\tau_2$
  are exact by assumption, and $\tau_1$ is furthermore bijective on
  objects.  It now follows from Lemma~\ref{lem:exexcons} that $g$ is
  exact.
\end{proof}

\begin{lemma}\label{lem:exexcons}
  Given a commutative triangle of symmetric strong monoidal functors
\[
\xymatrix{S \ar[rr]^f\ar[rd]_u && T  ,\\ & S' \ar[ru]_g &}
\]
  if $f$ is exact, and $u$ is exact and bijective on objects, then $g$
  is exact.
\end{lemma}
\begin{proof}
  By Theorem~\ref{thm:exactness-convolution-characterisations} it
  is enough to check that $g^*$ is strong monoidal.  Consider the
  corresponding triangle of pullback functors:
\[
\xymatrix{\widehat S && \ar[ll]_{f^*} \widehat T \ar[ld]^{g^*} \\
& \widehat S' \ar[lu]^{u^*} &}
\]
  All three functors are lax monoidal; $f^*$ and $u^*$ are strong
  monoidal because of exactness.  Furthermore $u^*$ is monadic since
  $u$ is bijective on objects, and so $u^*$ is conservative.  The
  monoidal coherences of $f^*$ are invertible; but these are obtained
  by applying $u^*$ to the lax coherence of $g^*$.  Since $u^*$ is
  conservative we can therefore conclude that already the coherences
  for $g^*$ must be invertible.
\end{proof}

\section{The hereditary condition and exactness}
\label{sec:Hereditary}

In this section we analyse the hereditary condition of Kaufmann and
Ward \cite{Kaufmann-Ward:1312.1269} and relate it to Guitart exactness
in Proposition~\ref{prop:rPat=hered}.  In Section~\ref{sec:pins} we
shall see that the hereditary condition is one of two conditions
characterising substitudes among pinned monoidal categories
(Proposition~\ref{prop:substitudes-among-pinned-SMCs}).
\begin{blanko}{Permutation-monotone factorisation.}
  As in Section \ref{sec:strict} for $n \in \mathbb{N}$, we denote by
  $\underline n$ the linearly-ordered set $\{1,...,n\}$.  Any function
  $\alpha : \underline m \to \underline n$ factors uniquely as
\[
\alpha = \lambda_\alpha \circ \sigma_\alpha,
\]
  where $\sigma_\alpha: \underline m \isopil \underline m$ is a
  permutation that is monotone on the fibre $\alpha^{-1}(j)$ for each
  $j\in \underline n$, and $\lambda_\alpha:\underline m \to \underline
  n$ is monotone{\footnotemark{\footnotetext{This factorisation is not
  part of a factorisation system, but it is nevertheless very
  useful.}}}.  With reference to $\alpha: \underline m \to \underline
  n$, if $(x_1,\ldots,x_m) = (x_i)_{i\in \underline m}$ is a sequence
  of objects, we denote by $(x_i)_{\alpha i = j}$ the subsequence
  consisting of those entries whose index maps to $j$.  The order is
  the induced order on the subset $\alpha^{-1}(j) \subset \underline
  m$.
\end{blanko}
\begin{blanko}{The hereditary condition.}\label{defn:hereditary}\label{hered}
  A symmetric colax monoidal functor $\tau : \freesym C \to M$ satisfies the 
  \emph{hereditary condition} when for all pairs of sequences 
  $(x_i)_{i\in\underline{m}}$ and $(y_j)_{j\in\underline{n}}$ of objects of $C$, 
  the function
  \[ \begin{array}{c}
  {\hmap_{\tau,x,y} : 
  \sum\limits_{\alpha :\underline{m}\to\underline{n}} \prod\limits_{j\in\underline{n}} M(\tau(x_i)_{\alpha i=j},\tau y_j) \longrightarrow M(\tau(x_i)_{i{\in}\underline{m}},\bigotimes_{j\in\underline{n}}\tau y_j)}
  \end{array} \]
  which sends $(\alpha,(g_j)_{j\in\underline{n}})$ to the composite
  \[ \xygraph{!{0;(2.75,0):(0,1)::} 
  {\tau(x_i)_{i{\in}\underline{m}}}="p0" [r] {\tau(x_{\rhomark_\alpha^{-1}i})_{i{\in}\underline{m}}}="p1" [r] {\bigotimes_{j{\in}\underline{n}}\tau(x_i)_{\alpha i=j}}="p2" [r] {\bigotimes_{j\in\underline{n}}\tau y_j}="p3"
  "p0":"p1"^-{\tau\sigmarhof}:"p2"^-{\overline{\tau}}:"p3"^-{\bigotimes_jg_j}} \]
  in $M$, is a bijection.  (Note that 
  $\bigotimes_{j{\in}\underline{n}}\tau(x_{\rhomark_\alpha^{-1}i})_{\lambda_\alpha i=j}
  = 
  \bigotimes_{j{\in}\underline{n}}\tau(x_i)_{\alpha i=j}$.)
  Note that the summation is taken over {\em arbitrary}
  functions $\alpha: \underline{m} \to \underline{n}$, not just monotone ones.

  In the case where $\tau$ is strict (i.e.~when $\overline{\tau}$ is the identity)
  one has $\tau(x_i)_{i{\in}\underline{m}} = \bigotimes_{i{\in}\underline{m}}\tau x_i$,
  and it may be more convenient to write $\hmap_{\tau,x,y}$ as the function
  \[ \begin{array}{c}
  {\sum\limits_{\alpha :\underline{m}\to\underline{n}} \prod\limits_{j\in\underline{n}} M(\bigotimes_{\alpha i=j}\tau x_i,\tau y_j) \longrightarrow M(\bigotimes_{i\in\underline{m}}\tau x_i,\bigotimes_{j\in\underline{n}}\tau y_j)}
  \end{array} \]
  which sends $(\alpha,(g_j)_{j\in\underline{n}})$ to
  the composite
  $$
  \underset{i\in\underline{m}}{\Tensor} \tau x_i
  \stackrel{\rhomark}\longrightarrow
  \underset{i\in\underline{m}}{\Tensor} \tau x_{\rhomark^{-1} i}
  =
  \underset{j\in\underline{n}}{\Tensor} \ \underset{\alpha i=j}{\Tensor} \tau x_i
  \stackrel{\tensor_j g_j}\longrightarrow
  \underset{j\in\underline{n}}{\Tensor} \tau y_j .
  $$

  In less formal terms, the hereditary condition says that every morphism $f$ of
  $M$ as on the right
  \[ \begin{array}{lcccr}
  {g_j : \bigotimes_{\alpha i=j}\tau x_i 
  \longrightarrow \tau y_j}
   &&&& 
   {f : \bigotimes_{i\in\underline{m}}\tau x_i \longrightarrow 
  \bigotimes_{j\in\underline{n}}\tau y_j}
  \end{array} \]
  can be uniquely decomposed as a tensor product of morphisms $g_j$ as on the left,
  modulo some symmetry coherence isomorphisms in $M$.  A useful slogan for this
  is: `many-to-many maps decompose uniquely as a tensor product of many-to-one
  maps'; which expresses the operadic nature of this condition.
  The hereditary condition has been discovered independently by various people
  in different guises.  While Kaufmann and Ward got it from
  Markl~\cite{Markl:0601129} via Borisov and Manin~\cite{Borisov-Manin:0609748},
  it is also equivalent to the (operad case of the) `operadicity' condition of
  Melli\`es and Tabareau~\cite[\S 3.2]{MelliesTabareau-TAlgTheoriesKan}.
\end{blanko}
\noindent The main result of this section is
\begin{prop}\label{prop:rPat=hered}
Let $\tau : \freesym C \to M$ be a symmetric colax monoidal functor.
\begin{parenenumerate}
\item If $\tau$ is exact then it satisfies the hereditary condition.
\label{propcase:exactness->hereditariness}
\item If $\tau$ is essentially surjective, strong monoidal and satisfies the hereditary condition, then $\tau$ is exact.
\label{propcase:hereditariness->exactness}
\end{parenenumerate}
\end{prop}
\noindent and its proof occupies the rest of this section.  For (1), we shall
first show that the sum-over-functions formula arises from the Day convolution
product, and second that the hereditary maps are special cases of the 
components of a canonical
$2$-cell associated to $\tau$.  For (2), we shall invoke a classical criterion
for exactness in terms of a category of factorisations, going back to Guitart
himself~\cite{Guitart:1980} in some form, and analysed in more detail in
\cite{Weber-Guitart}.
\begin{blanko}{Sums over functions from convolution for free symmetric monoidal categories.}
Our discussion begins by identifying how sums over functions, as in the domains of the hereditary condition maps $\hmap_{\tau,x,y}$, arise categorically. For a small category $C$ we define the functor
\[ \Asterisk : \freesym(\PSh{\freesym C}) \longrightarrow \PSh{\freesym C} \]
to be given on objects as
\begin{equation}\label{eq:convolutionForFreesFormula}
\begin{array}{rcl}
{(\opAst\limits_{j{\in}\underline{n}} F_j)(x_i)_{i{\in}\underline{m}}} & = &
{\sum\limits_{\alpha:\underline{m}\to\underline{n}} \prod\limits_{j{\in}\underline{n}} F_j(x_i)_{\alpha i=j}}
\end{array}
\end{equation}
where the sum is taken over all functions $\underline{m} \to \underline{n}$. We 
shall now see, as the notation chosen indicates, that $\Asterisk$ is the Day 
convolution tensor product.

Let us first exhibit the functoriality of 
$\opAst_j F_j$ in $(x_i)_{i{\in}\underline{m}}$
(that is, verify that the assignment in \eqref{eq:convolutionForFreesFormula} 
really defines a presheaf on $\freesym{C}$).
Given $(F_j)_{j{\in}\underline{n}}$ in $\freesym(\PSh{\freesym C})$ and a 
morphism $(\sigma,(\arrowinC_i)_i) : (x_i)_{i{\in}\underline{m}} \to
(\secondobjectinC_i)_{i{\in}\underline{m}}$ in $\freesym C$, note that the permutation $\sigma
\in \Sigma_m$ restricts to $\sigma_j : \alpha^{-1}(j) \to
(\alpha\sigma^{-1})^{-1}(j)$ for $j \in \underline{n}$, and so we get
$(\sigma_j,(\arrowinC_i)_{\alpha{i}=j}) : (x_i)_{\alpha{i}=j} \to (\secondobjectinC_i)_{\alpha{i}=j}$
in $\freesym C$ for each $j\in \underline{n}$.  Thus we define
$\smash{\opAst\limits_{j{\in}\underline{n}}}F_j(\sigma,(\arrowinC_i)_i)$ as the unique
function such that the square
\[ \xygraph{!{0;(2.75,0):(0,.5)::}
{\smash{\prod\limits_{j{\in}\underline{n}}} F_j(x_i)_{\alpha{i}=j}}="p0" [r] {(\smash{\opAst\limits_{j{\in}\underline{n}}} F_j)(x_i)_{i{\in}\underline{m}}}="p1" [d] {(\smash{\opAst\limits_{j{\in}\underline{n}}} F_j)(\secondobjectinC_i)_{i{\in}\underline{m}}}="p2" [l] {\smash{\prod\limits_{j{\in}\underline{n}}} F_j(\secondobjectinC_i)_{\alpha\sigma^{-1}{i}=j}}="p3"
"p0":"p1"^-{k_{\alpha}}:"p2"^-{\smash{\opAst\limits_j}F_j(\sigma,(\arrowinC_i)_i)}:@{<-}"p3"^-{k_{\alpha\sigma^{-1}}}:@{<-}"p0"^-{\smash{\prod\limits_j}(\sigma_j,(\arrowinC_i)_{\alpha{i}=j})}
} \]
commutes, where $k_{\alpha}$ and $k_{\alpha\rho^{-1}}$ are the sum inclusions.
With the functoriality of this assignment clear by definition, we have thus
defined the object map of $\Asterisk : \freesym(\PSh{\freesym C}) \to
\PSh{\freesym C}$.  We proceed to check that $\Asterisk$ is functorial.
Let $(\rho,(\mapofpresheaves_j)_j) : (F_j)_{j{\in}\underline{n}} \to
(G_j)_{j{\in}\underline{n}}$ be a morphism in $\freesym(\PSh{\freesym C})$.  
For any function $\alpha : \underline{m} \to \underline{n}$,
$(x_i)_{i{\in}\underline{m}}$ in $\freesym C$, and $j \in \underline{n}$, one
has the function $(\mapofpresheaves_j)_{(x_i)_{\alpha{i}=j}} : F_j(x_i)_{\alpha{i}=j} \to
G_{\rho{j}}(x_i)_{\alpha{i}=j}$.  Thus the components of
$\Asterisk(\rho,(\mapofpresheaves_j)_j)$ are defined by the commutativity of the squares
\[
\xygraph{!{0;(2.75,0):(0,.5)::}
{\smash{\prod\limits_{j{\in}\underline{n}}} F_j(x_i)_{\alpha{i}=j}}="q0" [r] {(\smash{\opAst\limits_{j{\in}\underline{n}}} F_j)(x_i)_{i{\in}\underline{m}}}="q1" [d] {(\smash{\opAst\limits_{j{\in}\underline{n}}} G_j)(x_i)_{i{\in}\underline{m}}}="q2" [l] {\smash{\prod\limits_{j{\in}\underline{n}}} G_{\rho{j}}(x_i)_{\alpha{i}=j}}="q3"
"q0":"q1"^-{k_{\alpha}}:"q2"^-{\Asterisk(\rho,(\mapofpresheaves_j)_j)_{(x_i)_i}}:@{<-}"q3"^-{k_{\rho\alpha}}:@{<-}"q0"^-{\smash{\prod\limits_j}(\mapofpresheaves_j)_{(x_i)_i}}
}
\]
for all $\alpha$ and $(x_i)_i$.  With the
functoriality of this assignment also clear by definition, we have thus defined
the functor $\Asterisk : \freesym(\PSh{\freesym C}) \to \PSh{\freesym C}$.
\end{blanko}
\begin{lemma}\label{lem:sumOverFnsFromDay}
For any small category $C$, the functor $\Asterisk : \freesym(\PSh{\freesym C})
\to \PSh{\freesym C}$ just defined describes the tensor product for Day
convolution on $\freesym C$.
\end{lemma}
\begin{proof}
The formula (\ref{eq:convolutionForFreesFormula}) is clearly colimit preserving in each $F_j$, and so it suffices to exhibit an isomorphism
\[ \xygraph{!{0;(2,0):(0,.5)::}
{\freesym^2C}="p0" [r] {\freesym(\PSh{\freesym C})}="p1" [d] {\PSh{\freesym C}}="p2" [l] {\freesym C}="p3"
"p0":"p1"^-{\freesym y_{\freesym C}}:"p2"^-{\Asterisk}:@{<-}"p3"^-{y_{\freesym C}}:@{<-}"p0"^-{\mu_C}
"p0":@{}"p2"|-*{\iso}} \]
because then, this isomorphism will exhibit $\Asterisk$ as a pointwise left Kan extension of $y_{\freesym C}\mu_C$ along $\freesym y_{\freesym C}$, giving the result by Proposition \ref{prop:Day-convolution}. In terms of the notation of Section \ref{sec:exact}, this isomorphism will then be the natural isomorphism $\overline{y}_{\freesym C}$ corresponding to our formula (\ref{eq:convolutionForFreesFormula}) for $\Asterisk$.

An object of $\freesym^2C$ is a sequence of sequences from $C$, which for
convenience we identify as a pair $(\psi,(c_i)_{i{\in}\underline{m}})$, where
$\psi : \underline{m} \to \underline{n}$ is monotone, and
$(c_i)_i \in \freesym C$.  Applying $y_{\freesym C}\mu_C$ to this gives the
representable $\freesym C(-,(c_i)_{i{\in}\underline{m}})$.  On the other hand,
for $(x_i)_{i{\in}\underline{l}} \in \freesym C$, the set
\begin{equation}\label{eq:DayOnRepresentables}
(\Asterisk \circ \freesym(y_{\freesym C}))(\psi,(c_i)_{i{\in}\underline{m}})(x_i)_{i{\in}\underline{l}}
\end{equation}
is, as with $\freesym
C((x_i)_{i{\in}\underline{l}},(c_i)_{i{\in}\underline{m}})$, the empty set when
$l \neq m$.  However when $l=m$, the set (\ref{eq:DayOnRepresentables}) is the
sum
\begin{equation}\label{eq:DayOnRepsNonTrivialCase}
\begin{array}{c}
{\smash{\sum\limits_{\alpha:\underline{m}\to\underline{n}}} \, \smash{\prod\limits_{j{\in}\underline{n}}} \, \freesym C((x_i)_{\alpha i=j},(c_i)_{\psi i=j}).}
\end{array}
\end{equation}
\vskip .5em 
\noindent 
To give an element of (\ref{eq:DayOnRepsNonTrivialCase})
is to give a function $\alpha : \underline{m} \to \underline{n}$, a permutation
$\sigma \in \Sigma_m$ such that $\psi\sigma = \alpha$ (which just says that $\sigma$
restricts to bijections $\sigma_j : \alpha^{-1}(j) \to \psi^{-1}(j)$), and for $i
\in \underline{m}$ an arrow $\arrowinC_i : x_i \to c_{\sigma{i}}$ of $C$.  This is
the same as to give a morphism $(\sigma,(\arrowinC_i)_i) : (x_i)_{i{\in}\underline{m}}
\to (c_i)_{i{\in}\underline{m}}$ of $\freesym C$ such that $\psi\sigma=\alpha$,
and this last equation shows that $\alpha$ is redundant.  We thus have our desired
isomorphism, whose naturality is very easy to check.
\end{proof}
\begin{blanko}{Hereditary condition maps as the components of a natural transformation.}
  We now return to the situation of a general symmetric colax monoidal functor
  $\tau : \freesym C \to M$.  Denote by $\betacell{\tau}$ the coherence 2-cell
  datum
  \[ 
  \xygraph{!{0;(2,0):(0,.5)::} {\freesym M}="p0" [r] {M}="p1" [d] {\PSh{\freesym C}}="p2" [l] {\freesym(\PSh{\freesym C})}="p3"
"p0":"p1"^-{\bigotimes}:"p2"^-{M(\tau,1)}:@{<-}"p3"^-{\Asterisk}:@{<-}"p0"^-{\freesym(M(\tau,1))}
"p0" [d(.55)r(.4)] :@{=>}[r(.2)]^-{\betacell{\tau}}} 
\]
for the symmetric lax monoidal functor $M(\tau,1) : M \to \PSh{\freesym C}$.
By Theorem~\ref{thm:exactness-convolution-characterisations}, exactness of 
$\tau$ is equivalent to the invertibility of $\betacell{\tau}$.

Using the explicit description of the Day convolution tensor product in
$\freesym(\PSh{\freesym C})$, just established in
Lemma~\ref{lem:sumOverFnsFromDay}, we see that the components of $\betacell{\tau}$
at $(w_j)_{j{\in}\underline{n}}$ in $\freesym M$ amount to maps of sets
\[ \begin{array}{c}
{(\betacell{\tau}_{(w_j)_j})_{(x_i)_i} : 
\sum\limits_{\alpha:\underline{m}\to\underline{n}}\prod\limits_{j{\in}\underline{n}}
M(\tau(x_i)_{\alpha{i}=j},w_j) 
\longrightarrow 
M(\tau(x_i)_i,\bigotimes\limits_j w_j)}
\end{array} \]
for $(x_i)_{i{\in}\underline{m}}$ in $\freesym C$, which  we proceed to describe.

\end{blanko}
\begin{lemma}\label{lem:descriptionOfBetaTau}
  The component $(\betacell{\tau}_{(w_j)_j})_{(x_i)_i}$ is the 
  function which sends 
  $(\alpha,(\gammatog_j)_{j{\in}\underline{n}})$
  to the composite
  \begin{equation}\label{eq:directDescBetaTau}
  \xygraph{
  *=(0,1.5)!(0,.3){\xybox{{\xygraph{!{0;(2.5,0):(0,.4)::}
  {\tau(x_i)_{i{\in}\underline{m}}}="p0" [r] {\tau(a_{\sigma_{\alpha}^{-1}i})_{i{\in}\underline{m}}}="p1" [r(1.2)] {\bigotimes_{j{\in}\underline{n}}\tau(x_i)_{\alpha{i}=j}}="p2" [r(1.25)] {\bigotimes_{j{\in}\underline{n}}\tau(w_j)}="p3"
  "p0":"p1"^-{\tau\sigma_{\alpha}}:"p2"^-{\overline{\tau}}:"p3"^-{\bigotimes_j\gammatog_j}}}}}}
  \end{equation}
\end{lemma}
  \vskip -2.25em \noindent
(Note that in the special case where $w_j = \tau y_j$ (that is, the components 
of $\betacell{\tau}$ in the image of $\tau$), we recover precisely the 
hereditary maps $\hmap_{\tau,x,y}$.)

\begin{proof}
  We saw in Section \ref{sec:exact} that $M(\tau,1)$ is the
pointwise left Kan extension of $y_{\freesym C}$ along $\tau$. It then follows 
from
Theorem 2.4.4(1) of \cite{Weber-Guitart} that $\betacell{\tau}$ can be 
characterised as the unique natural transformation satisfying the equation
\begin{equation}\label{eq:defBetaTau}
\xygraph{!{0;(1,0):(0,1.25)::} 
{\freesym^2C}="p0" [r(2)] {\freesym M}="p1" [dr] {M}="p2" [l(2)d] {\PSh{\freesym C}}="p3" [l(2)u] {\freesym C}="p4" [r(2)] {\freesym(\PSh{\freesym C})}="p5"
"p0":"p1"^-{\freesym\tau}:"p2"^-{\bigotimes}:"p3"^-{M(\tau,1)}:@{<-}"p4"^-{y_{\freesym C}}:@{<-}"p0"^-{\mu_C}
"p5" (:@{<-}"p0"|-{}="h0" "h0" [d(.15)l(.15)] {\scriptstyle{\freesym y_{\freesym C}}}, :@{<-}"p1"|-{}="h1" "h1" [d(.15)r(.4)] {\scriptstyle{\freesym M(\tau,1)}}, :"p3"^(.45){\Asterisk})
"p4" [r(.85)d(.1)] :@{=>}[r(.3)]^-{\overline{y}_{\freesym C}}
"p5" [r(.85)d(.1)] :@{=>}[r(.3)]^-{\betacell{\tau}}
"p5" [u(.6)l(.15)] :@{=>}[r(.3)]^-{\freesym\chi^{\tau}}
"p2" [r(1.25)] {=} [r(1.25)u]
{\freesym^2C}="q0" [r(2)] {\freesym M}="q1" [d] {M}="q2" [ld] {\PSh{\freesym C}}="q3" [lu] {\freesym C}="q4"
"q0":"q1"^-{\freesym\tau}:"q2"^-{\bigotimes}:"q3"^-{M(\tau,1)}:@{<-}"q4"^-{y_{\freesym C}} (:@{<-}"q0"^-{\mu_C}, :"q2"^-{\tau})
"q0" [d(.5)r(.85)] :@{=>}[r(.3)]^-{\overline{\tau}}
"q4" [d(.4)r(.85)] :@{=>}[r(.3)]^-{\chi^{\tau}}}
\end{equation}
  To prove the lemma it is therefore enough to check that
  the stated formula for $\betacell{\tau}$ satisfies Equation~\eqref{eq:defBetaTau}.
  As in the proof of Lemma \ref{lem:sumOverFnsFromDay} we denote an object of
  $\freesym^2C$ as a pair $(\psi,(c_i)_i)$, where 
  $\psi : \underline{m} \to \underline{n}$ is monotone and
  $(c_i)_{i{\in}\underline{m}}$
  is an object of $\freesym C$.  We check that the
  $(\psi,(c_i)_i)$-components of either side of (\ref{eq:defBetaTau}) agree.
Note that for $(x_i)_{i{\in}\underline{l}}$ in 
$\freesym C$, the $((x_i)_{i{\in}\underline{l}}, (\psi,(c_i)_i))$-components of
both sides of (\ref{eq:defBetaTau}) give functions
\[ \begin{array}{c}
{\freesym C((x_i)_{i{\in}\underline{l}},(c_i)_{i{\in}\underline{m}}) \longrightarrow M(\tau(x_i)_{i{\in}\underline{l}},\bigotimes_j\tau(c_i)_{\psi{i}=j}).}
\end{array} \]
The domain of these functions is empty when $l \neq m$, and so it suffices to
consider the case when $l = m$.  Let $(\rho,(\arrowinC_i)_{i{\in}\underline{m}})$
be in $\freesym C((x_i)_{i{\in}\underline{m}},(c_i)_{i{\in}\underline{m}})$.
Note that for $j \in \underline{n}$, the permutation $\rho$ restricts to
bijections $\rho_j : (\psi\rho)^{-1}(j) \to (\psi)^{-1}(j)$, and so one has
$(\rho_j,(\arrowinC_i)_{\psi{i}=j})$ in $\freesym
C((x_i)_{\psi\rho{i}=j},(c_i)_{\psi{i}=j})$.  Applying the $((x_i)_i,
(\psi,(c_i)_i))$-component of the left hand side of (\ref{eq:defBetaTau}) to
$(\rho,(\arrowinC_i)_{i{\in}\underline{m}})$ gives the composite
(\ref{eq:directDescBetaTau}) with $\gammatog_j =
\tau(\rho_j,(\arrowinC_i)_{\psi{i}=j})$.  On the other hand, applying the
$((x_i)_i, (\psi,(c_i)_i))$-component of the right hand side to
$(\rho,(\arrowinC_i)_{i{\in}\underline{m}})$ gives the composite
$\overline{\tau}_{(\psi,(c_i)_i)} \circ \tau(\rho,(\arrowinC_i)_i)$.  These
coincide by the naturality of $\overline{\tau}$.
\end{proof}
\begin{proof}[of Proposition~\ref{prop:rPat=hered}~(\ref{propcase:exactness->hereditariness})]
  In view of Lemma \ref{lem:descriptionOfBetaTau}, the hereditary condition for
  $\tau$ says that the natural transformation $\betacell{\tau}S(\tau\eta_C)$ is
  invertible.  The result follows since by 
  Theorem~\ref{thm:exactness-convolution-characterisations}, exactness is equivalent to
  the invertibility of $\betacell{\tau}$.
\end{proof}
The argument just given is not immediately adaptable to prove the converse.
Even if $\tau$ is essentially surjective, the functor $S(\tau\eta_C) : \freesym
C \to \freesym M$ is not, and so even for essentially surjective $\tau$, it is
not apparent that the invertibility of $\betacell{\tau}S(\tau\eta_C)$ implies the
invertibility of $\betacell{\tau}$.  Hence the need for the further assumption on
$\tau$ of strong monoidality among the hypotheses of Proposition
\ref{prop:rPat=hered}(\ref{propcase:hereditariness->exactness}).  Moreover,
it seems to us that the clearest proof of this result follows from the
consideration of factorisation categories, to which we now turn.
\begin{blanko}{Factorisation category of a morphism $f: Fv \to \tensor_k w_k$.}
Consider a symmetric colax monoidal functor $F: V \to W$ (with coherences $\overline F: (\tensor_i v_i) \to \tensor_i F v_i$). For each morphism in $W$ of the form $f: Fv \to \tensor_k w_k$, we define the category of factorisations $\operatorname{Fact}(f)$ as having objects diagrams of the form
\[ \xymatrix{
Fv \ar[rrr]^f\ar[rd]_{Fh} &&& \tensor_k w_k  .  \\ 
& F( \tensor_k u_k) \ar[r]  & \tensor_k F(u_k) \ar[ru]_{\tensor_k g_k} &} 
\]
Morphisms in $\operatorname{Fact}(f)$ are connecting maps (in $V$) between the $u_k$ and the $u'_k$ making a left-hand triangle commute already in $V$, and making the right hand triangle commute in $W$. (The middle square commutes by naturality of the coherence data.)
\end{blanko}
By Proposition 4.5.5(1) of \cite{Weber-Guitart}, $F$ is exact as a
symmetric colax monoidal functor ($=$~colax morphism of
$\freesym$-algebras) if and only if it is exact as a colax monoidal functor
($=$~colax morphism of $\freemon$-algebras).  Applying Lemma
\ref{lem:combinatorial-exactness-Cat-fact-version} to the lax square
containing $F$'s colax $\freemon$-morphism datum, gives the following
explicit characterisation of exact symmetric colax monoidal functors.
\begin{lemma}\label{lem:Fact}
A symmetric colax monoidal functor $F: V \to W$ as above is exact
if and only if for every $f:v \to \tensor_k w_k$ the category 
$\operatorname{Fact}(f)$ is connected.  \qed
\end{lemma}
When $V$ is $\freesym C$, and $F$ is strict monoidal $\tau: \freesym C \to M$,
the description of $\operatorname{Fact}(f)$ simplifies considerably.
In this case $v= (x_i)_{i\in \underline m}$, and 
each $u_k$ is a sequence 
$(z_i)_{i\in \underline m_k}$.  Furthermore, since $h$ is a map in $\freesym C$
it is a labelled permutation, so in particular the concatenated sequence
$((z_i)_{i\in \underline m_k})_{k\in \underline r}$ is of the same length as 
$\underline m$, so altogether we can write a factorisation of $f: \tau(x_i)_{i\in 
\underline m} \to \tensor_{k\in \underline r} w_k$ as
$$\xymatrix{
\tau(x_i)_{i\in \underline m} \ar[r]^-{\tau(h)} & 
\tau( (z_i)_{\beta i=k} )_{k \in \underline r} \ = \ 
\tensor_{k \in \underline r} \tensor_{\beta i=k}
\tau z_i
\ar[r]^-{\tensor_{k\in \underline r} g_k} &
\tensor_{k\in \underline r} w_k  ,
}$$
where $\beta: \underline m \to \underline r$ is monotone.
Since $h$ is a labelled permutation, we can keep the pure permutation 
part as a first factor, and then absorb the second factor (the `labels')
into the $g_k$ on the right.  Renaming those $g_k$ accordingly, we arrive at a factorisation
$$\xymatrix{
\tau(x_i)_{i\in \underline m} \ar[r]^-{\sim} & 
\tau( (x_i)_{\beta i=k} )_{k \in \underline r} \ = \ 
\tensor_{k \in \underline r} \tensor_{\beta i=k}
\tau x_i
\ar[r]^-{\tensor_{k\in \underline r} g_k} &
\tensor_{k\in \underline r} w_k  .
}$$
Factorisations of this form we call {\em normalised}.  Clearly every factorisation receives a morphism from its normalisation.
\begin{proof}[of Proposition~\ref{prop:rPat=hered}~(\ref{propcase:hereditariness->exactness})]
  By strictification, we can assume that $\tau$ is strict monoidal and identity on objects. We fix 
  $f: \tau(x_i)_{i \in \underline m} = \tensor_{i \in \underline m} \tau x_i
  \to \tensor_{k\in \underline r} w_k$
  and aim to show that $\operatorname{Fact}(f)$ is connected.
  The identity-on-object condition means that each
  $w_k$ is a tensor product of certain $\tau y_j$, say $w_k = 
  \tensor_{\oldlambda j = k} \tau y_j$, where $\oldlambda: \underline n \to 
  \underline r$ is monotone.  Altogether,
  $$
  \tensor_{k \in \underline r} w_k = \tensor_{k \in \underline r} 
  \tensor_{\oldlambda j = k} \tau y_j = \tensor_{j \in \underline n} \tau y_j.
  $$
  The map 
  $f: \tensor_{i \in \underline m} \tau x_i \to \tensor_{k\in \underline r} w_k$
  is now of the form
  $$
  f: \tensor_{i \in \underline m} \tau x_i \to \tensor_{j\in \underline n} \tau y_j ,
  $$
  and by the hereditary condition we get a factorisation as
  $$\xymatrix{
\tensor_{i \in \underline m} \tau x_i \ar[r]^-{\sim} & 
\tensor_{j \in \underline n} \tensor_{\alpha i=j}
\tau x_i
\ar[r]^-{\tensor_{j\in \underline n} s_j} &
\tensor_{j\in \underline n} \tau y_j ,
}$$
  which can also be written as
  \begin{equation}\label{eq:standardfact}
    \xymatrix{
    \tensor_{i \in \underline m} \tau x_i \ar[r]^-{\sim} & 
    \tensor_{k\in \underline r}\tensor_{\oldlambda j=k} \tensor_{\alpha i=j}
    \tau x_i
    \ar[rr]^-{\tensor_{k \in \underline r} (\tensor_{\oldlambda j =k} s_j)} &&
    \tensor_{k \in \underline r}(\tensor_{\oldlambda j =k} \tau y_j) .
    }
    \end{equation}
  This we refer to as the standard factorisation of $f$.  It is seen to
  be an object in $\operatorname{Fact}(f)$ by putting
  $g_k  = \tensor_{\oldlambda j =k} s_j$.
  In particular we have now shown that $\operatorname{Fact}(f)$ is not empty.
  
  Given now any other normalised object in $\operatorname{Fact}(f)$
  \begin{equation}\label{eq:givenfact}
    \xymatrix{
    \tensor_{i \in \underline m} \tau x_i \ar[r]^-{\sim} & 
    \tensor_{k \in \underline r} \tensor_{\beta i=k}
    \tau x_i
    \ar[r]^-{\tensor_{k\in \underline r} g'_k} &
    \tensor_{k\in \underline r} w_k ,
    }
  \end{equation}
  where $\beta: \underline m \to \underline r$ monotone,
  we would like to connect it to the standard factorisation just constructed.
  To this end we apply the hereditary condition to each of the maps 
  $$
  g'_k : \tensor_{\beta i=k} \tau x_i
  \to
  w_k = \tensor_{\oldlambda j = k} \tau y_j .
  $$
  This gives us a function $\alpha_k : \beta^{-1}(k) \to \oldlambda^{-1}(k)$,
  and all these functions assemble into a function $\alpha: \underline m \to \underline n$
  such that $\oldlambda \circ \alpha = \beta$.
  With reference to these $\alpha_k$, the hereditary condition gives us
  a normalised factorisation of $g_k$ as
  $$
  \xymatrix{
  \tensor_{\beta i=k} \tau x_i
  \ar[r]^-\sim &
  \tensor_{\oldlambda j = k} \tensor_{\alpha i = j} \tau x_i 
  \ar[rr]^-{\tensor_{\oldlambda j = k} s'_j} &&
  \tensor_{\oldlambda j = k} \tau y_j
  }
  $$
  so that altogether $f$ factors as
    $$\xymatrix{
  \tensor_{i \in \underline m} \tau x_i \ar[r]^-{\sim} & 
  \tensor_{k \in \underline r} \tensor_{\beta i=k} \tau x_i
  \ar[r]^-\sim &
  \tensor_{k\in \underline r} \tensor_{\oldlambda j = k} \tensor_{\alpha i = j} \tau 
  x_i
  \ar[rr]^-{\tensor_{k\in \underline r} \tensor_{\oldlambda j = k} s'_k} &&
  \tensor_{k\in \underline r} \tensor_{\oldlambda j = k} \tau y_j .
  }$$
  If we take the middle permutation to belong to the left-hand factor,
  then we obtain a factorisation of the standard shape \eqref{eq:standardfact},
  and by the uniqueness property
  in the hereditary condition, this must actually be {\em equal} to 
  the standard factorisation \eqref{eq:standardfact}, that is $s'_j = s_j$ for all $j\in \underline n$.
  On the other hand, if we let the middle permutation belong to the right-hand
  factor, we get precisely the given normalised factorisation \eqref{eq:givenfact}.
  Hence the given normalised factorisation is connected to the standard 
  factorisation.
  Since we already remarked that any factorisation is connected to its 
  normalisation, we have altogether shown that $\operatorname{Fact}(f)$ is connected.
\end{proof}
An alternative proof can be derived
from the results in Section~\ref{sec:pins}:
by 
Proposition \ref{prop:substitudes-among-pinned-SMCs}, we may assume that
$\tau$ is of the form $\LHerm(\phi)$, where $\phi : C \to P$ is a substitude.  
It can be checked directly that for any morphism of operads $\phi$, the
symmetric monoidal functor $\LHerm(\phi)$ is exact.  See \cite{Weber-Guitart}
(Corollary 3.4.1) for a proof.

\section{Kaufmann--Ward comma-category condition}

\label{sec:comma}

\begin{blanko}{Feynman categories.}
Kaufmann and Ward~\cite{Kaufmann-Ward:1312.1269} define a {\em Feynman category} to
be a symmetric strong monoidal functor $\tau: \freesym C \to M$ satisfying the
three conditions
  \begin{parenenumerate}
    \item $C$ is a groupoid
    
    \item $\tau$ induces an equivalence of groupoids $\freesym C \isopil 
    M_{\operatorname{iso}}$

    \item $\tau$ induces an equivalence of groupoids
    $\freesym(M\comma C)_{\operatorname{iso}} 
    \isopil (M\comma M)_{{\operatorname{iso}}}$.
  \end{parenenumerate}
In this section we show that the comma-category condition (3), can be
reformulated in terms of the hereditary map
\[ \begin{array}{c}
  {\hmap := \hmap_{\tau,x,y} : 
  \sum\limits_{\alpha :\underline{m}\to\underline{n}} \prod\limits_{j\in\underline{n}} M(\bigotimes_{\alpha i=j}\tau x_i,\tau y_j) \longrightarrow M(\bigotimes_{i\in\underline{m}}\tau x_i,\bigotimes_{j\in\underline{n}}\tau y_j)}
  \end{array}
\]
from \ref{defn:hereditary}.  This is more or less
implicit in \cite{Kaufmann-Ward:1312.1269}.  This section does not seem to
generalise to the enriched setting.
\end{blanko}

\begin{prop}\label{FC-comma}
For an essentially surjective symmetric strong monoidal functor
$\tau:\freesym C \to M$, the following are equivalent:
  
  \begin{parenenumerate}
    \item $\tau:\freesym C\to M$ is hereditary and
    $\tau_{\operatorname{iso}}: \freesym C_{\operatorname{iso}} \to M_{\operatorname{iso}}$ 
    is an equivalence of groupoids
  
    \item The natural map $\oldw : \freesym(M\comma C)_{\operatorname{iso}} 
    \to (M \comma M)_{\operatorname{iso}}$ is 
    an equivalence of groupoids.

  \end{parenenumerate}
\end{prop}
In view of Proposition~\ref{prop:rPat=hered}, this shows:
\begin{cor}\label{cor:FC}
  A Feynman category is precisely a regular
  pattern $\tau:\freesym C \to M$ for which $C$ is a groupoid and
  $\tau_{\operatorname{iso}}: \freesym C_{\operatorname{iso}} \to
  M_{\operatorname{iso}}$ is an equivalence of groupoids.
  \qed
\end{cor}

\begin{proof}[of Proposition~\ref{FC-comma}]
  We immediately reduce to the strict situation, where $\tau$ is identity-on-objects.
  Lemma~\ref{lem:ff} below says that if just $\oldw$ is fully faithful, then also
  $\tau_{\operatorname{iso}}$ is fully faithful (and therefore actually an 
  isomorphism).  So we can separate that out as a global assumption.
  Now $\oldw$ is full if and only if $\hmap$ is injective by Lemma~\ref{lem:inj},
  and $\oldw$ is essentially surjective if and only if $\hmap$ is surjective by 
  Lemma~\ref{lem:surj}.  Finally, it is actually automatic that $\oldw$ is 
  faithful (again by Lemma~\ref{lem:ff}).
\end{proof}

\begin{lemma}\label{lem:ff}
  If $\oldw:\freesym(M\comma C)_{\operatorname{iso}} \to (M \comma
  M)_{\operatorname{iso}}$ is full, respectively faithful, then
  $\tau_{\operatorname{iso}} : \freesym C_{\operatorname{iso}} \to
  M_{\operatorname{iso}}$ is full, respectively faithful.  Conversely, if
  $\tau_{\operatorname{iso}}$ is faithful then $\oldw$ is faithful.
\end{lemma}
\begin{proof}
  The first statements follow immediately from the commutative diagram
    $$\xymatrix{
     \freesym C_{\operatorname{iso}} \ar[r]^{\tau_{\operatorname{iso}}}\ar[d] & M_{\operatorname{iso}} \ar[d] \\
     \freesym(M\comma C)_{\operatorname{iso}} \ar[r]_{\oldw} & (M\comma M)_{\operatorname{iso}}
  }$$
  since the vertical maps are fully faithful.
  
  For the last statement,
  the isomorphisms in the image of $\oldw$ are of the form
    $$\xymatrix{
     \bigotimes_j \tensor_i x_i \ar[r]^a\ar[d]_{\tensor_j g_j} 
     & \bigotimes_j \tensor_i x'_i \ar[d]^{\tensor_j g'_j} \\
     \tensor_j y_j \ar[r]_b & \tensor_j y'_j   .
  }$$
  If individually the horizontal isos can arise from $\freesym 
  C$ in at most one way, then taken together it is even harder to arise from
  $\freesym (M\comma C)$.
\end{proof}

\begin{lemma}\label{lem:surj}
  For an identity-on-objects symmetric strict monoidal functor
  $\tau:\freesym C \to M$, such that
  $\freesym C_{\operatorname{iso}} \simeq M_{\operatorname{iso}}$, the following
  are equivalent:  
  \begin{parenenumerate}
    \item The map $\hmap$ in the hereditary condition is surjective.
  
    \item The natural map $\oldw:\freesym(M\comma C)_{\operatorname{iso}} \to (M 
    \comma M)_{\operatorname{iso}}$ is essentially surjective.
  \end{parenenumerate}
\end{lemma}

\begin{proof}
  Throughout the proof we suppress $\tau$ on objects, since anyway $\tau$ is
  identity-on-objects.
  
  We first prove that if $\hmap$ is surjective, then $\oldw$ is essentially
  surjective.  Given some object in $M\comma M$, that's precisely an element in
  the codomain of $\hmap$, say $f: \tensor_i x_i \to \tensor_j y_j$.  By
  surjectivity of $\hmap$, there is an element $(\alpha, g_1,\ldots,g_n)$
  in the domain of $\hmap$, whose tensor product is $f$.  Now just the
  sequence $(g_1,\ldots,g_n)$ is an object in $\freesym(M\comma C)$ and by
  assumption, their tensor product is isomorphic to $f$ (by the permutation
  $\sigma_\alpha$ obtained from $\alpha$).
  
  Conversely, assuming $\oldw$ is essentially surjective, let us prove that
  $\hmap$ is surjective.  Given $f: \tensor_i x_i \to \tensor_j
  y_j$, an element in the codomain of $\hmap$, we need to
  construct an element on the left---that's a tuple $(\alpha,
  g_1,\ldots,g_n)$---such that the composite
  $$
  \xymatrix{
  \tensor_i x_i \ar[r]^-{\sigma_\alpha} & \tensor_j \tensor_k x_k \ar[r]^-{\tensor_j g_j}
  & \tensor_j y_j
  }$$
  is equal to $f$.  (Here $\sigma_\alpha$ is the permutation part of $\alpha$, obtained
  by permutation-monotone factorisation).
  Since $\oldw$ is essentially surjective, there exists an object $(\lambda', 
  h'_1,\ldots,h'_n)$
  in $\freesym(M\comma C)$ (here $\lambda': \underline m \to \underline n$ is 
  monotone, and
  $h'_j : \tensor_{\lambda' i'=j'} x'_{i'} \to y'_{j'}$)
  whose tensor product is isomorphic
  to $f$:
  $$\xymatrix{
     \tensor_i x_i \ar[r]_-\sim^-a\ar[d]_f & \tensor_{j'} 
     \tensor_{\lambda' i'=j'} x'_{i'} 
     \ar[d]^{\tensor_{j'} h'_{j'}} \\
     \tensor_j y_j \ar[r]^-\sim_-b & \tensor_{j'} y'_{j'} .
  }$$
  Since $\tau_{\operatorname{iso}}$ is full, both $a$ and $b$ come from
  $\freesym C$, 
  and in particular can be written
  as a permutation followed by a tensor product of isos in $C$.
  Let $\varsigma$ be the permutation underlying $a$ and let $\rho$ be the permutation
  underlying $b$.  Put
  $$
  \alpha := \rho^{-1} \circ \lambda' \circ \varsigma .
  $$
  The permutation $\rho: \underline n \isopil \underline n$ is
  such that with $j'=\rho j$ we have $y_j \simeq y'_{j'}$.   Conjugating with
  this isomorphism we find
  $$\xymatrix{
     \tensor_j \tensor_{\lambda' i' = \rho j} x'_{i'} 
     \ar[r]^{\tilde b}\ar[d]_{\tensor_j g'_j} &
     \tensor_{j'} \tensor_{\lambda' i' = j'} x'_{i'} 
     \ar[d]^{\tensor_j h'_{j'}} \\
     \tensor_j y_j \ar[r]_-b & \tensor_{j'} y'_{j'} .
  }$$
  (The permutation $\tilde \rho: \underline m \isopil \underline m$ underlying 
  $\tilde b$ permutes
  the blocks according to the permutation $b$.)
  Except for some isos in $C$, the new map $g'_j$ is essentially $h'_{\rho j}$.
  The tensor product of the new maps $g'_j$ now have underlying indexing map 
  $\lambda := \rho^{-1} \circ \lambda' \circ \tilde \rho$, which is different from 
  $\lambda'$,
  but is still monotone.
  On the other hand, the permutation $\varsigma: \underline m \isopil \underline m$ is
  such that with $i'=\varsigma i$ we have $x_i \simeq x'_{i'}$.  Using this, we
  can rewrite the upper left-hand corner
  $$
  \tensor_j \tensor_{\lambda i' = \rho j} x'_{i'}  \simeq
  \tensor_j \tensor_{\lambda \varsigma i = \rho j} x_i  =
  \tensor_j \tensor_{\alpha i = j} x_i   ,
  $$
  but a little care is needed with this substitution, since it may permute stuff inside
  each $j$-factor.  However, this permutation can be absorbed into each $g'_j$
  and now called $g_j$ (and this does not affect $\lambda$), giving altogether
  $$\xymatrix{
     \tensor_i x_i \ar[r]^-\sigma\ar[d]_f 
     & \tensor_j \tensor_{\alpha i=j} x_i \ar[d]^{\tensor_j g_j} \ar[r]
     & \tensor_{j'} \tensor_{\lambda'i'=j'} x'_{i'} \ar[d]^{\tensor_{j'} h'_{j'}} \\
     \tensor_j y_j  \ar[r]_-=
     &\tensor_j y_j \ar[r]
     & \tensor_j y'_{j'}   .
  }$$
  The remaining permutation $\sigma$ is monotone on $\lambda$-fibres by
  construction, and since $\alpha = \lambda \circ \sigma$, we see that
  $(\alpha,g_1,\ldots,g_n)$ is a solution to our problem: 
  by construction,
  $\hmap$ applied to $(\alpha,g_1,\ldots,g_n)$
  is the original $f$.  Hence $\hmap$ is surjective.
\end{proof}

\begin{lemma}\label{lem:inj}
  For an identity-on-objects
  symmetric strict monoidal functor $\tau:\freesym C \to M$, such that
  $\freesym C_{\operatorname{iso}} \simeq M_{\operatorname{iso}}$, the following
  are equivalent:
  \begin{parenenumerate}
    \item The map $\hmap$ in the hereditary condition is injective.
  
    \item The natural map $\oldw:\freesym(M\comma C)_{\operatorname{iso}} \to (M 
    \comma M)_{\operatorname{iso}}$ is full.
  \end{parenenumerate}
\end{lemma}

\begin{proof}
  Throughout the proof we suppress $\tau$ on objects, since anyway $\tau$ is
  identity-on-objects.
  
  Let us show that if $\oldw$ is full then $\hmap$ is injective.
  Suppose we have two elements in the domain of $\hmap$ both giving $f$.  Say
  $(\alpha,g_j)$ and $(\alpha',g'_j)$.  (In both cases $j$ runs to the same $n$:
  that's part of the data in $f$).  This gives us now a commutative diagram of
  maps in $M$:
  $$\xymatrix{
  \bigotimes_j \tensor_i x_i \ar[d]_{\tensor_j g_j} & \ar[l]_-{\sigma} \tensor_i x_i \ar[d]^f
  \ar[r]^-{\sigma'} & \bigotimes_j \tensor_i x_i \ar[d]^{\tensor_j g'_j} \\
  \tensor_j y_j  & \ar@{=}[l] \tensor_j y_j \ar@{=}[r] & \tensor_j y_j
  }$$
  The outer vertical maps,  $\tensor_j g_j$ and $\tensor_j g'_j$, are now two objects in 
  $M\comma M$, exhibited isomorphic by means of $\sigma'\circ 
  \sigma^{-1}$ and the identity at the bottom. 
  Since $\oldw$ is full, 
  this isomorphism comes from one in $\freesym (M\comma C)$, hence is given by a labelled
  permutation of $\underline n$.  Since the bottom map is the identity, also the
  permutation $\sigma'\circ \sigma^{-1}$ is the identity on the outer tensor factors,
  those 
  indexed by $j$.  So $\sigma$ and $\sigma'$ agree on outer factors.
  But they are also monotone on fibres, so in fact they must agree completely.
  It follows that $g_j = g'_j$, and hence in particular also that
  $\alpha=\alpha'$.
    
  Conversely, assuming that $\hmap$ is injective, let us show that $\oldw$ is 
  full.
  Given two objects in $M\comma M$ in the image of $\freesym(M\comma C)$, pictured
  vertically, and an iso between them (consisting of two isomorphisms, pictured
  horizontally):
  $$\xymatrix{
     \bigotimes_j \tensor_i x_i \ar[r]^a\ar[d]_{\tensor_j g_j} 
     & \bigotimes_j \tensor_i x'_i \ar[d]^{\tensor_j g'_j} \\
     \tensor_j y_j \ar[r]_b & \tensor_j y'_j
  }$$
  A priori, we don't know that the $j$ run to the same $n$, but in fact they do,
  because of the existence of $b$: since $\tau_{\operatorname{iso}}$ is fully
  faithful, $b$ is in fact the image of an isomorphism in $\freesym C$,
  which is to say that it is a
  labelled permutation.  For the same reason, $a$ is a labelled permutation
  too.  Since the vertical maps are in the image of $\oldw$, their
  corresponding $\alpha$ and $\alpha'$ have trivial permutation part.  It
  follows that the permutation $a$ must actually be a refinement of the
  permutation $b$.  In particular we have necessarily $\alpha=\alpha'$, which
  is a monotone map, since it just comes from the tensor product.
  
  It remains to check that $a$ and $b$ together actually form a valid isomorphism in 
  $\freesym(M\comma C)$.  We have found that the original square is the tensor
  product of a sequence of squares of the form
  $$\xymatrix{
     \tensor_{i\in \alpha^{-1}\sigma^{-1}j} x_i \ar[r]^a\ar[d]_{g_j} 
     & \tensor_{i\in \alpha^{-1}(j)} x'_i \ar[d]^{g'_j} \\
     y_{\sigma^{-1}j} \ar[r]_b & y'_j
  }$$
  It remains to check that each of them commutes.
  For fixed $j\in \underline n$, the two composites in this individual square
  are both elements in the set
  $$
  M( \!\!\!\!\!\!\bigotimes_{i \in \alpha^{-1}\sigma^{-1}(j)}\!\!\!\!\!\! x_i, \ y'_j) ,
  $$
  so altogether they form a tuple which is an element in
  $$
  \prod_{j\in \underline n}
  M( \!\!\!\!\!\!\bigotimes_{i \in \alpha^{-1}\sigma^{-1}(j)}\!\!\!\!\!\! x_i, \ y'_j).
  $$
  That's in the $\sigma\circ \alpha$ summand of the domain of $\hmap$.
  If one of the squares did not commute, it would thus constitute two
  distinct elements in this set, with the same image under tensoring
  (the map $\hmap$).  Since $\hmap$ is injective, we conclude that
  in fact all those squares do commute, and hence form a valid isomorphism
  in $\freesym(M\comma C)$, as required.
\end{proof}

\section{Hermida-type adjunctions and the main theorem}
\label{sec:pins}

\begin{blanko}{Operads.}
By {\em operad} we always mean symmetric coloured operad (in $\Set$), 
also known as symmetric multicategory.  Hence, an operad $P$ consists 
of a set $I$ of objects (also called colours), for each pair 
$(x_1,...,x_n;y)$ consisting of a finite sequence of 
{\em input objects} $(x_1,...,x_n)$ and one {\em output object}
$y$, a set $P(x_1,...,x_n;y)$ of operations from $(x_1,\ldots,x_n)$ 
to $y$.  Moreover there is an identity operation 
$1_x : (x) \to x$ in $P(x;x)$, and a substitution law satisfying 
the usual axioms.
  
We shall use various shorthand notation for sequences $(x_1,\ldots,x_n)$,
such as $(x_i)_{i \in \underline n}$, or just $(x_i)_i$, or even 
$\mathbf x$, when practical.
    
Operads form a $2$-category $\Opd$, the morphisms being morphisms 
of operads in the usual sense.  A $2$-cell
\[ \xygraph{{P}="p0" [r(2)] {Q}="p1"
"p0":@/^{1pc}/"p1"^-{f}|-{}="t" "p0":@/_{1pc}/"p1"_-{g}|-{}="b"
"t":@{}"b"|(.25){}="d"|(.75){}="c" "d":@{=>}"c"^{\omega}} \]
consists of components $\omega_x : (fx) \to gx$, required to be natural
with respect to all operations of $P$, that is, given an operation
$\someoperation : (x_i)_i \to
y$, one has $\omega_y f(\someoperation) = g(\someoperation)(\omega_{x_i})_i$.
  
A small category can be regarded as an operad with only unary operations,
and in this way $\Cat$ becomes a coreflective sub-$2$-category of $\Opd$
(the coreflector picks out the unary part of an operad.)  Various notions
from category theory makes sense also for operads: in particular, a
morphism of operads $f: P \to Q$ is called {\em fully faithful} if it
induces bijections on multihomsets $P(\mathbf{x};y) \isopil Q(f\mathbf{x};
f(y))$.  Gabriel factorisation works the same for operads as for
categories, as exploited also in \cite{Gambino-Joyal:1405.7270}: every
operad morphism factors as bijective-on-objects followed by fully faithful.
Just as in the category case, this is an enhanced factorisation system.
And just as in the category case, one can always choose the
bijective-on-objects part to be actually identity-on-objects (with which
choice the factorisation is unique).
\end{blanko}

\begin{blanko}{Hermida adjunctions.}
  For any symmetric monoidal category $M$, its {\em endomorphism operad}
  $\End(M)$ has objects those of $M$, and sets of operations given by
  \[
  \begin{array}{lcr} 
    {\End(M)((x_1,...,x_n);y)} & = & {M(\bigotimes_{i=1}^n x_i,y).} 
  \end{array}
  \]
  The category $\Alg P(M)$, of $P$-algebras in $M$, is defined as
  \[
  \Alg P(M) := \Opd(P,\End(M))
  \]
  which at the level of objects, says that a $P$-algebra in $M$ is a
  morphism of operads $P \to \End(M)$.  The assignment $M \mapsto \End(M)$
  is the effect on objects of a 2-functor $\End : \SymMonCat \to \Opd$,
  where $\SymMonCat$ denotes the 2-category of symmetric monoidal
  categories, symmetric strong monoidal functors, and monoidal natural
  transformations.  Below, we will also consider the restriction of this
  2-functor to a 2-functor $\sSymMonCat \to \Opd$, where $\sSymMonCat$ is
  the locally full sub-2-category of $\SymMonCat$ consisting of the
  symmetric strict monoidal categories and symmetric strict monoidal
  functors.

  In the other direction, one can associate to any operad $P$, a symmetric
  strict monoidal category $\LHerm P$, which has the following explicit
  description.  An object of $\LHerm P$ is a finite sequence of objects of
  $P$.  A morphism $(x_1,...,x_m) \to (y_1,...,y_n)$ in $\LHerm P$,
  consists of an \emph{indexing function} $\alpha : \underline{m}\to
  \underline{n}$, together with for each $j \in \underline{n}$, an
  operation $\someoperation_j : (x_i)_{\alpha i=j} \to y_j$.  The category structure
  of $\LHerm P$ comes from substitution in $P$ and the composition of
  (indexing) functions.  Thus the homs of $\LHerm P$ are given by
  \[
  \LHerm P\big((x_i)_{i\in\underline m},(y_j)_{j\in\underline n}\big)
  = \sum_{\alpha:\underline{m}\to\underline{n}}\prod_{j{\in}\underline{n}}
     P((x_i)_{\alpha i=j};y_j).
  \]  
  Note that for a category $C$ regarded as an operad, we have a 
  canonical identification $\LHerm C = \freesym C$.

Now, $\LHerm P$ enjoys a strict universal property amongst 
all symmetric strict monoidal categories, expressed by the 
isomorphisms of categories on the left
\[ \begin{array}{lcccr}
{\sSymMonCat(\LHerm P,M) \iso \Opd(P,\End(M))} &&&
{\SymMonCat(\LHerm P,M) \simeq \Opd(P,\End(M))}
\end{array} \]
2-naturally in $M$; and also a bicategorical universal
property amongst all symmetric monoidal categories 
expressed by the equivalences of categories on the right, 
which are pseudo-natural in $M$. Taken together one thus 
has a 2-adjunction as indicated on the left
\[ \xygraph{!{0;(3,0):(0,1)::} {\sSymMonCat}="p0" [r] {\Opd}="p1"
"p0":@<-1ex>"p1"_-{\End}|-{}="t":@<-1ex>"p0"_-{\LHerm}|-{}="b" "t":@{}"b"|-{\perp}
"p1" [r] {\SymMonCat}="q0" [r] {\Opd}="q1"
"q0":@<-1.2ex>"q1"_-{\End}|-{}="t":@<-1.2ex>"q0"_-{\LHerm}|-{}="b" "t":@{}"b"|-{\perp_{\tn{b}}}} \]

\vspace*{-8pt}

\noindent
and a biadjunction as indicated on the right. We shall call these the \emph{Hermida adjunctions} in honour of Claudio Hermida, who studied the strict version in the non-symmetric case \cite{Hermida-RepresentableMulticategories} (see also \cite{Elmendorf-Mandell:0710.0082}, Theorem 4.2). In \cite{Weber-CodescCrIntCat} Corollary 6.4.7, they were obtained formally from the $2$-monad $\freesym$.

Recall that to give a biadjunction with left adjoint $\LHerm$ and right adjoint $\End$ is to give pseudo-natural transformations $\eta^H : 1_{\Opd} \to \End\LHerm$ and $\varepsilon^H : \LHerm\End \to 1_{\SymMonCat}$ together with invertible modifications
\[ \xygraph{!{0;(1.5,0):(0,.6667)::} 
{\LHerm}="p0" [r(2)] {\LHerm\End\LHerm}="p1" [dl] {\LHerm}="p2"
"p0":"p1"^-{\LHerm\eta^H}:"p2"^-{\varepsilon^H\LHerm}:@{<-}"p0"^-{1_{\LHerm}} "p0" [d(.45)r] {\iso}
"p1" [r(1.5)]
{\End}="q0" [r(2)] {\End\LHerm\End}="q1" [dl] {\End}="q2"
"q0":"q1"^-{\eta^H\End}:"q2"^-{\End\varepsilon^H}:@{<-}"q0"^-{1_{\End}}
"q0" [d(.45)r] {\iso}} \]
sometimes called the left and right \emph{triangulators} for the 
biadjunction. In our case the 2-adjunction can be recovered from the 
biadjunction by restricting from $\SymMonCat$ to $\sSymMonCat$. In terms of the unit and counit, this says the following: (1) the unit of the Hermida biadjunction is the same as that of the 2-adjunction and so is strictly 2-natural; (2) the left triangulator is an identity; (3) for morphisms of $\sSymMonCat$ the associated $\varepsilon^H$ pseudo-naturality datum is an identity; and (4) for objects of $\sSymMonCat$ the associated component of the right triangulator is an identity. We turn now to an explicit description of the components of $\eta^H$ and $\varepsilon^H$.

The counit component at $M \in \SymMonCat$
\[ \begin{array}{lccr}
{\varepsilon_{M}^{H} : \LHerm(\End(M)) \longrightarrow M} &&&
{(x_i)_i \mapsto \bigotimes_i x_i} \end{array} \]
has object map as indicated, and hom functions
\[ \begin{array}{c} 
\LHerm(\End(M)) \big((x_i)_{i\in\underline{m}},(y_j)_{j\in\underline{n}}\big) 
\longrightarrow 
M\big(\bigotimes_{i\in\underline{m}}x_i,\bigotimes_{j\in\underline{n}}y_j \big) 
\end{array} \]
which send
\[ \begin{array}{c} {(\alpha : \underline{m} \to \underline{n},(g_j : 
\bigotimes_{\alpha i=j}x_i \to y_j)_j)} \end{array} \]
to the composite
\[ \underset{i\in\underline{m}}{\Tensor} x_i \stackrel{\sigma}\longrightarrow \underset{i\in\underline{m}}{\Tensor} x_{\sigma^{-1} i} \iso
\underset{j\in\underline{n}}{\Tensor} \ \underset{\alpha i=j}{\Tensor} x_i \stackrel{\tensor_j g_j}\longrightarrow \underset{j\in\underline{n}}{\Tensor} y_j \]
where $\sigma=\sigma_\alpha$ is the permutation part of $\alpha$ (given by the permutation/monotone factorisation), and the unnamed isomorphism is obtained from the coherences for $M$, these being identities when $M$ is strict.

The component of the unit at $P \in \Opd$ is given on objects as
\begin{eqnarray*}
\eta^H_P : P & \longrightarrow & \End(\LHerm P)  \\
x & \longmapsto & (x) .
\end{eqnarray*}
More interesting is to see what $\eta^H$ does on sets of operations: it is a map
\[ P(x_1,\ldots,x_m;y) \to \End(\LHerm P)\big((x_1),\ldots,(x_m);(y)\big) \]
but the last set we can unravel as
\[  = \LHerm P \big( (x_1,\ldots,x_m) , (y) \big) \]
since the tensor product in $\LHerm P$ is just concatenation of sequences;
\[ = \sum_{\alpha: \underline m \to \underline 1} \prod_{j\in \underline 1} P((x_i)_{\alpha i = j} ; y_j) \]
by definition of hom sets in $\LHerm P$;
\[ = P(x_1,\ldots,x_m;y) \]
since there exists only one indexing map $\underline m\to \underline 1$. In the end, $\eta^H_P$ is the identity map on operations. 
In conclusion:
\end{blanko}

\begin{lemma}\label{lem:Hermida-unit-ff}
The components of $\eta^H$, the unit for the Hermida adjunction,
are fully faithful operad maps.\qed
\end{lemma}

\begin{blanko}{Pinned operads.}\label{defn:pinned-operads}
  We use the term {\em pinned} for an object (e.g.~a symmetric monoidal 
  category, an operad, or a substitude) equipped with a map singling 
  out some objects (the pins).
  
  A \emph{pinned operad} is a triple $(C,\phi,P)$, in which $C$ is a 
  category regarded as an operad with only unary operations, $P$ is an 
  operad, and $\phi : C \to P$ is a morphism of operads. Pinned operads 
  assemble into a $2$-category $\pOpd$. A morphism $(C_1,\phi_1,P_1) \to 
  (C_2,\phi_2,P_2)$ is a triple $(f,g,\omega)$ consisting of a functor 
  $f : C_1 \to C_2$, an operad morphism $g : P_1 \to P_2$, and an invertible 
  2-cell $\phi_2 f \to g\phi_1$ in $\Opd$. A 2-cell $(f,g,\omega) \to 
  (f',g',\omega')$ is a pair $(\alpha,\beta)$, where 
  $\alpha : f \to f'$ is a natural transformation, and 
  $\beta : g \to g'$ is a 2-cell of $\Opd$, such that
\[ \xygraph{!{0;(1.5,0):(0,.6667)::} 
{C_1}="p0" [r] {P_1}="p1" [d] {P_2}="p2" [l] {C_2}="p3"
"p0":"p1"^-{\phi_1}:@/^{1pc}/"p2"^-{g'}:@{<-}"p3"^-{\phi_2}:@/^{1pc}/@{<-}"p0"^-{f} "p0":@/^{1pc}/"p3"^-{f'}
"p0" [d(.55)r(.7)] :@{=>}[r(.2)]^-{\omega'}
"p0" [d(.55)l(.1)] :@{=>}[r(.2)]^-{\alpha}
"p1" [d(.5)r(.8)] {=} [u(.5)r(.75)]
{C_1}="q0" [r] {P_1}="q1" [d] {P_2}="q2" [l] {C_2}="q3"
"q0":"q1"^-{\phi_1}:@/^{1pc}/"q2"^-{g'}:@{<-}"q3"^-{\phi_2}:@/^{1pc}/@{<-}"q0"^-{f} "q1":@/_{1pc}/"q2"_-{g}
"q0" [d(.55)r(.2)] :@{=>}[r(.2)]^-{\omega}
"q1" [d(.55)l(.1)] :@{=>}[r(.2)]^-{\beta}} \]
  in $\Opd$. Compositions for $\pOpd$ are inherited from $\Opd$ in 
  the obvious way. A morphism $(f,g,\omega)$ of $\pOpd$ is said to 
  be \emph{strict} when $\omega$ is an identity 2-cell, and the wide 
  and locally full sub-2-category of $\pOpd$ consisting of the strict 
  morphisms is denoted $\spOpd$. The 2-categories $\pSymMonCat$ and 
  $\spSymMonCat$ of pinned symmetric monoidal categories, and pinned 
  symmetric strict monoidal categories, were described above in~\ref{bl:pSMC}.
\end{blanko}

\begin{blanko}{Pinned Hermida adjunctions.}\label{const:pinned-Hermida}
The Hermida adjunctions have pinned analogues
\[ \xygraph{!{0;(3,0):(0,1)::} {\spSymMonCat}="p0" [r] {\spOpd}="p1"
"p0":@<-1ex>"p1"_-{\pEnd}|-{}="t":@<-1ex>"p0"_-{\pLHerm}|-{}="b" "t":@{}"b"|-{\perp}
"p1" [r] {\pSymMonCat}="q0" [r] {\pOpd}="q1"
"q0":@<-1ex>"q1"_-{\pEnd}|-{}="t":@<-1ex>"q0"_-{\pLHerm}|-{}="b" "t":@{}"b"|-{\perp_{\tn{b}}}} \]

\vspace*{-8pt}

\noindent
which we now describe. The object maps of $\pEnd$ and $\pLHerm$ are
\[ \xygraph{{\xybox{\xygraph{{\freesym C} :[r] {M}^-{\tau} [r(.5)] 
:@{|->}[r(.5)] [r(.5)] {C} :[r(1.5)] {\End(\freesym C)}^-{\eta^H_{C}} :[r(2)]
{\End(M)}^-{\End(\tau)}}}} [r(7)]
*!(0,.025){\xybox{\xygraph{*+!(0,-.025){C} :[r] *+!(0,-.025){P}^-{\phi} 
[r(.5)] :@{|->}[r(.5)] [r(.6)] {\freesym C} :[r(1.5)] {\LHerm P}^-{\LHerm(\phi )}}}}} \]
respectively, and since $\End$ and $\LHerm$ are $2$-functors and $\eta^H$
is $2$-natural, these definitions extend in the obvious way to arrows
and $2$-cells.
   
The component of the unit $\eta^{\tn{p}}$ of 
$\pLHerm \isleftadjointto_b \pEnd$ at $(C,\phi,P)$ is $(1_{C},\eta^H_P,\id)$, 
as depicted in the diagram on the left
\[ \xygraph{!{0;(2.5,0):(0,.4)::}
{C}="p0" [r(1.5)] {P}="p1" [d] {\End(\LHerm P)}="p2" [l(.9)] 
{\End(\freesym C)}="p3" [l(.6)] {C}="p4" "p0":"p1"^-{\phi }:"p2"^-{\eta^H_P}:@{<-}"p3"^-{\End(\LHerm(\phi ))}:@{<-}"p4"^-{\eta^H_C}:@{=}"p0"
"p1" [r]
{\freesym C}="q0" [r(.7)] {\LHerm(\End(\freesym C))}="q1" [r(1.1)] 
{\LHerm(\End(M))}="q2" [d] {M}="q3" [l(1.8)] {\freesym C}="q4" "q0":"q1"^-{\LHerm(\eta^H_{C})}:"q2"^-{\LHerm(\End(\tau))}:"q3"^-{\varepsilon^H_{M}}:@{<-}"q4"^-{\tau}:@{=}"q0"
"q1":"q4"^(.4){\varepsilon^H_{\freesym C}}
"q0" [d(.6)r(1.2)] {\iso} [u(.25)] {\scriptstyle{\varepsilon^H_{\tau}}}} \]
which commutes by the naturality of $\eta^H$. So the components of $\eta^{\tn{p}}$ all live in $\spOpd$. The component of the counit $\varepsilon^{\tn{p}}$ of $\pLHerm \isleftadjointto_b \pEnd$ at $(C,\tau,M)$ is $(1_C,\varepsilon^H_M,\varepsilon^H_{\tau}\LHerm(\eta^H_C))$, as depicted on the right in the previous display. Note that when $(C,\tau,M)$ is strict, the pseudo-naturality datum $\varepsilon^H_{\tau}$ is an identity, and then this component of $\varepsilon^{\tn{p}}$ lives in $\spOpd$. The pseudo-naturality data for the unit and counit, and the left and right triangulators for $\pLHerm \isleftadjointto_b \pEnd$ are inherited from $\LHerm \isleftadjointto_b \End$. The manner in which the biadjunction restricts to a 2-adjunction is clearly inherited also.
\end{blanko}

\begin{blanko}{Substitudes \cite{Day-Street:substitudes}.}\label{defn:set-substitute}
The notion of substitude was introduced by Day and Street~\cite{Day-Street:substitudes} as a general setting for substitution. It is a common generalisation of operad and (symmetric) monoidal category. See the appendix of \cite{Batanin-Berger-Markl:0906.4097} for a concise account of the basic theory. A substitude is like an operad, but allowing for a category of objects instead of just a set of objects. The data is
  
  -- a category $C$
  
  -- a functor $ P : (\freesym C)\op\times C \to \Set $
  
  -- composition and unit laws, subject to non-surprising axioms.
  
Substitudes can be described also as monads in the bicategory of generalised species~\cite{Fiore-Gambino-Hyland-Winskel:JLMS}. For the present purposes, the most convenient is to package the definition into the following (cf.~\cite[6.3]{Day-Street:lax-monoids} for the equivalence between the two formulations of the definition): A \emph{substitude} is a pinned operad $\phi : C \to P$ in which $\phi$ is the identity on objects. We denote by $\Subst$ the full sub-2-category of $\pOpd$ consisting of the substitudes, and denote by $\JJ : \Subst \to \pOpd$ the inclusion.
\end{blanko}

\begin{blanko}{Substitude coreflection.}\label{corefl}
Since the bijective-on-objects morphisms form the left class of an
enhanced factorisation system (Gabriel factorisation), by general principles, the inclusion functor $\JJ : \Subst \to \pOpd$ has a right adjoint, denoted $(-)'$, forming a 2-adjunction
\[ \xygraph{!{0;(3,0):(0,1)::}
{\pOpd}="p0" [r] {\Subst .}="p1"
"p0":@<-1ex>"p1"_-{(-)'}|-{}="t":@<-1ex>"p0"_-{\JJ}|-{}="b" "t":@{}"b"|-{\perp}} \]

\vspace*{-8pt}

\noindent
Explicitly, given a pinned operad $(C,\phi,P)$, factoring $\phi$ as identity-on-objects then fully faithful as on the left in
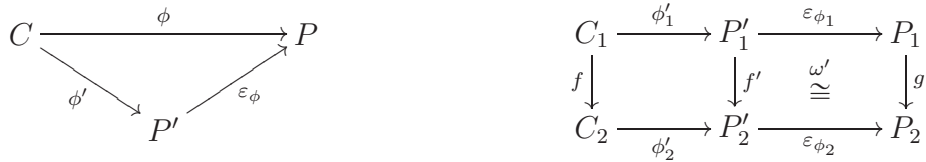
\[ \xygraph{!{0;(1.5,0):(0,.6667)::} 
{C}="p0" [r(2)] {P}="p1" [dl] {P'}="p2"
"p0":"p1"^-{\phi }:@{<-}"p2"^-{\varepsilon_\phi}:@{<-}"p0"^-{\phi'}
"p1" [r(2)]
{C_1}="q0" [r] {P'_1}="q1" [r(1.2)] {P_1}="q2" [d] {P_2}="q3" [l(1.2)] {P'_2}="q4" [l] {C_2}="q5"
"q0":"q1"^-{\phi'_1}:"q2"^-{\varepsilon_{\phi_1}}:"q3"^-{g}:@{<-}"q4"^-{\varepsilon_{\phi_2}}:@{<-}"q5"^-{\phi'_2}:@{<-}"q0"^-{f} "q1":"q4"^{f'} "q1" [d(.6)r(.6)] {\iso} [u(.25)] {\scriptstyle{\omega'}}} \]
one obtains the substitude $(C,\phi',P')$ together with the $(C,\phi,P)$-component of the counit of $\JJ \isleftadjointto (-)'$, which we denote as $\varepsilon_\phi$. Given a morphism $(f,g,\omega) : (C_1,\phi_1,P_1) \to (C_2,\phi_2,P_2)$ of $\pOpd$, one induces $f'$ and $\omega'$ uniquely so that the composite on the right in the previous display is $\omega$, using the enhancedness of $\Opd$'s Gabriel factorisation. Using the fact that this Gabriel factorisation is also $\Cat$-enriched, one can easily exhibit $(-)'$'s 2-cell mapping explicitly.

It is suggestive to write the operad part of $\phi': C \to P'$ as $P|C$: it is the operad base-changed to its pins.
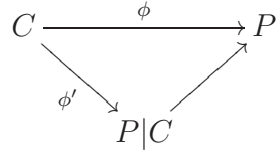
\[ \xymatrix{C \ar[rr]^\phi\ar[rd]_{\phi'} && P  \\ & P|C \ar[ru] &} \]
The operad $P|C$ has the same objects as $C$ by construction, and operations
\[ P|C( x_1,\ldots,x_m;y) = P(\phi x_1,\ldots,\phi x_m; \phi y) . \] \vspace*{-8pt}
\end{blanko}

\begin{blanko}{The substitude biadjunction.}
Now compose the pinned Hermida biadjunction \ref{const:pinned-Hermida}
with the coreflection of substitudes into pinned operads \ref{corefl}:
\[ \xygraph{!{0;(3,0):(0,1)::}
{\pSymMonCat}="p0" [r] {\pOpd}="p1" [r] {\Subst}="p2"
"p0":@<-1ex>"p1"_-{\pLHerm}|-{}="t1":@<-1ex>"p0"_-{\pEnd}|-{}="b1" "t1":@{}"b1"|-{\perp_b}
"p1":@<-1ex>"p2"_-{(-)'}|-{}="t2":@<-1ex>"p1"_-{\JJ}|-{}="b2" "t2":@{}"b2"|-{\perp}} \]

\vspace*{-4pt}

\noindent
The composed biadjunction
\begin{equation}\label{eq:pSMC-Subst-adjunction}
  \begin{gathered}
\xygraph{!{0;(3,0):(0,1)::} {\pSymMonCat}="p0" [r] {\Subst}="p1"
"p0":@<-1ex>"p1"_-{\iopEnd}|-{}="t":@<-1ex>"p0"_-{\iopLHerm}|-{}="b" "t":@{}"b"|-{\perp_b}}
\end{gathered}
\end{equation}

\vspace*{-2pt}

\noindent
goes like this: the left adjoint takes a substitude $C \to P$ to the 
pinned symmetric monoidal category $\freesym C \to \LHerm P$, and the right adjoint
takes a pinned symmetric monoidal category $\freesym C \to M$ to $C \to \End(M)|C$.

Unlike the Hermida biadjunction and the pinned version, this one has invertible unit:
\end{blanko}

\begin{prop}\label{prop:unit-iop-iso}
The unit for the substitude biadjunction $\iopLHerm \isleftadjointto \iopEnd$ is invertible.
\end{prop}
\begin{proof}
Let $\phi : C \to P$ be a substitude. In the diagram
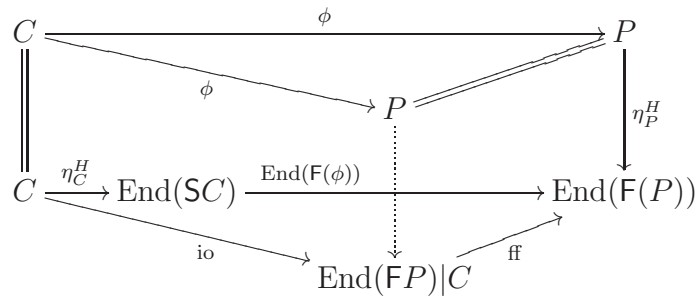
\[ \xymatrix @R=15pt {
C \ar[rrr]^\phi\ar@{=}[dd] \ar[rrd]_\phi &&& P \ar[dd]^{\eta^H_P} \\
&& P \ar@{=}[ru]  \ar@{..>}[dd] & \\
C \ar[r]^-{\eta^H_C} \ar[rrd]_{\mathrm{io}}& \End(\freesym C) 
\ar[rr]^(0.3){\End(\LHerm (\phi))} && \End(\LHerm (P)) \\
&& \End(\LHerm P)|C \ar[ru]_{\mathrm{ff}} &} \]
the back rectangle is the unit $\eta^{\mathrm{p}}_\phi$. The unit we are after, $\eta^{\mathrm{iop}}_\phi$, is obtained by Gabriel factorising its horizontal arrows: this induces the dotted arrow by functoriality of the factorisation, and the left-hand square is now $\eta^{\mathrm{iop}}_\phi$.  It is clearly invertible if and only if the dotted arrow is invertible.  But the dotted arrow is invertible because it compares two Gabriel factorisations of $\End(\LHerm (\phi)) \circ \eta^H_C$: on one hand $\phi$ followed by $\eta^H_P$ (the latter being fully faithful by Lemma~\ref{lem:Hermida-unit-ff}) and on the other hand the factorisation at the bottom of the diagram.
\end{proof}

\begin{blanko}{Counit of the substitude biadjunction.}
Thanks to Proposition~\ref{prop:unit-iop-iso} the $2$-category $\Subst$ is biequivalent to the full sub-$2$-category of $\pSymMonCat$, spanned by those pinned symmetric monoidal categories $(C,\tau,M)$ for which the the counit $\varepsilon^{\tn{iop}}_\tau $ is an equivalence. In the remainder of this section, we determine for which $(C,\tau,M)$ this is the case. Fundamental to this characterisation is the hereditary condition \ref{hered}. We begin the discussion by characterising equivalences in $\pSymMonCat$.
\begin{lemma}\label{lem:pSMC-equivalences}
A morphism $(f,g,\omega) : (C_1,\tau_1,M_1) \to (C_2,\tau_2,M_2)$ of $\pSymMonCat$ is an equivalence if and only if the functors $f$ and $g$ are equivalences of categories.
\end{lemma}
\begin{proof}
It suffices to show that the forgetful 2-functor
\[ \begin{array}{lccr} {\pSymMonCat \longrightarrow \Cat \times \Cat} &&& {(C,\tau,M) \mapsto (C,M)} \end{array} \]
reflects equivalences. So we suppose that $(f,g,\omega) : (C_1,\tau_1,M_1) \to (C_2,\tau_2,M_2)$ is a morphism of $\pSymMonCat$ such that $f$ and $g$ are equivalences of categories. Choose adjoint pseudo inverses $f'$ and $g'$ of $f$ and $g$ respectively, and then note that $g'$ admits a unique symmetric strong monoidal structure making the adjoint equivalence it participates in in $\Cat$, into one in $\SymMonCat$. Thus we have adjoint equivalences $(\freesym f,\freesym f')$ and $(g,g')$ in $\SymMonCat$, and so one can take the mate $\omega' : \tau_1\freesym(f') \iso g'\tau_2$ of $\omega$. This makes $(f',g',\omega')$ into a morphism of $\pSymMonCat$. Since $\omega$ and $\omega'$ are mates via the adjoint equivalences $(\freesym f,\freesym f')$ and $(g,g')$, these assemble to an adjoint equivalence in $\pSymMonCat$, exhibiting $(f',g',\omega')$ as an adjoint pseudo-inverse of $(f,g,\omega)$.
\end{proof}
We now unpack the counit component $\varepsilon^{\tn{iop}}_\tau$ for $(C,\tau,M)$ in $\pSymMonCat$. First we describe the substitude
$\iopEnd(\tau)$.  It is obtained by Gabriel factorising the pinned operad $\pEnd(\tau)$, which is the top composite here:
\[
\xymatrix{C \ar[r]^-{\eta^H_C}\ar[rd]_{\iopEnd(\tau)} & \End(\freesym C) \ar[r]^-{\End(\tau)} & \End(M) \\
& \End(M)|C \ar[ru]_{\epsilon}^{\mathrm{ff}}  &}
\]
Note that the operad $\End(M)|C$ has objects those of $C$, and
operation sets
\[
\End(M)|C \big(x_1,\ldots,x_m; y\big) 
= \End(M)\big(\tau x_1,\ldots,\tau x_m; \tau y\big) 
= M \big(\tau x_1 \tensor \cdots \tensor \tau x_m , \tau y\big).
\]
Now apply $\iopLHerm$, rendered as the top square in the next diagram;
the diagram as a whole is the counit $\varepsilon^{\tn{iop}}_\tau$,
the key part being of course the right-hand vertical composite:
\begin{equation}\label{eq:eiop}
  \begin{gathered}
\xygraph{!{0;(2,0):(0,.5)::} 
{\freesym C}="p0" [r(2.5)]
{\LHerm(\End(M)|C)}="p1" [d(1.2)]
{\LHerm\End(M)}="p2" [d(1.2)]
{M}="p3" [l(2.5)]
{\freesym C}="p4" [u(1.2)]
{\freesym C}="p5" [r]
{\LHerm\End(\freesym C)}="p6"
"p0":"p1"^-{\LHerm\iopEnd(\tau)}:"p2"^-{\LHerm(\epsilon)}:"p3"^-{\varepsilon^H_M}:@{<-}"p4"^-{\tau}
"p5" 
(
:@{=}"p4", :"p6"^(.35){\LHerm\eta^H_{\freesym 
C}}(:"p2"^-{\LHerm\End(\tau)}, :"p4"^(.4){\varepsilon^H_{\freesym C}}))
"p6" (
[d(.8)r(.7)] {\iso} [u(.3)] {\scriptstyle{\varepsilon^H_{\tau}}})
"p5":@{=}"p0"
}\end{gathered}
\end{equation}

\begin{prop}\label{prop:substitudes-among-pinned-SMCs}
  For a pinned symmetric monoidal 
  category $\tau : \freesym C \to M$,
  the counit component $\varepsilon^{\tn{iop}}_\tau$ is an 
  equivalence if and only if $\tau$ is essentially surjective and satisfies 
  the hereditary condition \ref{hered}.
\end{prop}

\begin{proof}
  By Lemma \ref{lem:pSMC-equivalences}, $\varepsilon^{\tn{iop}}_\tau$
  is an equivalence in $\pSymMonCat$ if and only if
  $\varepsilon^H_M\LHerm(\epsilon)$ is an equivalence of categories.
  Since $\LHerm\iopEnd(\tau)$ is an identity on objects, 
  Diagram~\eqref{eq:eiop} shows that $\varepsilon^H_M\LHerm(\epsilon)$ is
  essentially surjective if and only if $\tau$ is.  So we have reduced 
  to the case where $\tau$ is essentially surjective, which is the
  content of the next lemma.
\end{proof}
\begin{lemma}\label{lem:hered=ff}
  An essentially surjective pinned symmetric monoidal 
  category $\tau : \freesym C \to M$ satisfies the
  hereditary condition if and only if $\varepsilon^H_M\LHerm(\epsilon)$ is fully
  faithful.
\end{lemma}
\begin{proof}
  Applying Power coherence, we can reduce to the
  case where $\tau$ is an identity-on-objects symmetric strict
  monoidal functor.
  Let us now unpack the effect on morphisms of
  $\varepsilon^H_M\LHerm(\epsilon)$ in this case.  The objects of the
  symmetric monoidal category $\LHerm(\End(M)|C)$ are sequences of
  objects in $C$, and $\varepsilon^H_M\LHerm(\epsilon)$ sends
  $(x_1,\ldots,x_m)$ to $\tau x_1 \tensor \cdots \tensor \tau x_m$.
  The hom sets of $\LHerm( \End(M)|C )$ are
\[ \LHerm(\End(M)|C) \big( (x_i)_{i\in \underline m}, (y_j)_{j\in\underline n} \big) = \sum\limits_{\alpha:\underline{m}\to\underline{n}} \prod\limits_{j\in\underline{n}} M \big(\bigotimes_{\alpha i=j}\tau x_i,\tau y_j\big) , \]
and the hom mapping of $\varepsilon^H_M\LHerm(\epsilon)$ into $M (\tensor_i \tau x_i, \tensor_j \tau  y_j)$, sends $(\alpha, g_1,\ldots,g_n)$ to the composite
\[ \underset{i\in\underline{m}}{\Tensor} \tau x_i \stackrel{\sigma} \longrightarrow
\underset{i\in\underline{m}}{\Tensor} \tau x_{\sigma^{-1} i}  =
\underset{j\in\underline{n}}{\Tensor} \ \underset{\alpha i=j}{\Tensor} \tau x_i \stackrel{\tensor_j g_j}
\longrightarrow
\underset{j\in\underline{n}}{\Tensor} \tau y_j . \]
Thus, the hom functions for 
$\varepsilon^H_M\LHerm(\epsilon)$ are exactly the functions $\hmap_{\tau,x,y}$ 
whose bijectivity is the hereditary condition.
\end{proof}
Taking Theorem~\ref{thm:exactness-convolution-characterisations} 
and Propositions~\ref{prop:rPat=hered} and \ref{prop:substitudes-among-pinned-SMCs} together, we have thus arrived at the first part of the main theorem:
\end{blanko}

\begin{thm}\label{thm:Im(F)}
  $\iopLHerm : \Subst \to \pSymMonCat$ induces a biequivalence $\Subst
  \sim \RPat$.  \qed
\end{thm}

\begin{blanko}{Operads.}
  There are three ways of regarding operads as substitudes.  For a
  given operad $P$, the options are: the trivial pinning, where $C =
  \operatorname{obj}(P)$, the discrete category of objects; the
  canonical groupoid pinning, where $C = P_1^{\operatorname{iso}}$,
  the groupoid of invertible unary operations; and the full pinning,
  where $C= P_1$, the category of all unary operations.

  We are interested at the moment in the canonical groupoid pinning,
  $P_1^{\operatorname{iso}} \to P$.  Note that the functor
  $(-)^{\operatorname{iso}} : \Cat \to \Grpd$ is \emph{not} a
  2-functor.  However it does make sense to apply
  $(-)^{\operatorname{iso}}$ to \emph{invertible} natural
  transformations, and so $(-)^{\operatorname{iso}}$ is
  $\Grpd$-enriched.  If, for any 2-category $\ca K$, we denote by
  $(\ca K)^{\operatorname{2-iso}}$ its underlying groupoid-enriched
  category, obtained from $\ca K$ simply by ignoring non-invertible
  2-cells, then another way to express these considerations is to say
  that one has a 2-functor $(-)^{\operatorname{iso}} :
  (\Cat)^{\operatorname{2-iso}} \to \Grpd$.

  Thus, the process of taking the canonical groupoid pinning is the
  effect on objects of a 2-functor
  \[
  \GPin : (\Opd)^{\operatorname{2-iso}} \longrightarrow \Subst
  \]
  which is 2-fully-faithful when the codomain is restricted to
  $(\Subst)^{\operatorname{2-iso}}$.  Writing $\FeynCat$ for the full
  sub-2-category of $(\pSymMonCat)^{\operatorname{2-iso}}$ consisting
  of the Feynman categories in the sense of Kaufmann and Ward, the
  second part of our main theorem is as follows.
\end{blanko}

\begin{thm}\label{thm:FC=opd}
  $\iopLHerm \circ \GPin$ restricts to a biequivalence 
  $(\Opd)^{\operatorname{2-iso}} \sim \FeynCat$.
\end{thm}
\begin{proof}
  First we check that a pinned symmetric monoidal category in the
  image is a Feynman category, using \ref{cor:FC}.  We already know
  that it satisfies the hereditary condition, and by construction
  $P_1^{\operatorname{iso}}$ is a groupoid.  It remains to check that
  \[
  \freesym P_1^{\operatorname{iso}} \to (\LHerm P)_{\operatorname{iso}}
  \]
  is an equivalence.  These two categories have the same objects,
  namely sequences of objects in $P$.  The arrows in $\LHerm P$ from
  $\mathbf{x}$ to $\mathbf{y}$ form the hom set
  \[
  \sum_{\alpha:\underline m \to \underline n} \prod_{j\in \underline n} 
  P((x_i)_{i\in \alpha^{-1}(j)} ; y_j)
  \]
  They are composed as operations in $P$.  There is an obvious
  forgetful functor to $\Set$, given by returning the indexing set for
  the sequence.  For an arrow to be invertible, at least its
  underlying $\alpha$ must be a bijection, and furthermore it is then
  clear that the involved operations have to be invertible unary
  operations.
  
  Conversely, if we start with a Feynman category $\tau: \freesym C
  \to M$, the image in $\Subst$ is the substitude $C \to \End(M)|C$,
  and since $C$ is a groupoid and $\freesym C\simeq
  M_{\operatorname{iso}}$, it is clear that the invertible unary
  operations of $\End(M)|C$ are precisely those of $C$, which is the
  condition for the substitude to be in the image of $\GPin$.
\end{proof}

\begin{blanko}{Two other subcategories of $\kat{RPat}$ equivalent to the category of operads.}
  Let us note that there are two other subcategories of $\kat{RPat}$
  which are equivalent to the category of operads, given by the two other
  natural embeddings of $\Opd$ into $\Subst$: one takes an operad $P$
  to its discrete pinning $\obj(P)\to P$---this is a 
  $1$-functor only, as it is not defined on $2$-cells.  The other 
  embedding takes
  $P$ to its full pinning $P_1 \to P$ (this is a honest $2$-functor).
  
  The first corresponds to the subcategory of discrete substitudes, which in
  turn corresponds to the subcategory of regular patterns $\freesym C\to M$ for
  which $C$ is discrete.  
  While this is obviously a stronger condition than the
  first Feynman-category condition, that of being a groupoid, on the other hand
  the second Feynman-category condition, $\freesym C \isopil
  M_{\operatorname{iso}}$ is not in general satisfied.  Imposing this second
  condition on top of the discreteness condition gives the notion of discrete
  Feynman category, mentioned in \cite{Kaufmann-Ward:1312.1269} to be related to
  operads.  More precisely this notion corresponds to operads in which there are
  no unary operations other than the identities (the locus of operads for which
  the groupoid-pinning and discrete-pinning embeddings coincide).  
  (This embedding of the $1$-category of operads into regular 
  patterns was also observed by Getzler~\cite{Getzler:0701767}.)
  
  The second embedding of operads into substitudes, endowing an operad $P$ with
  its full pinning $P_1 \to P$, has as essential image that of {\em normal}
  substitudes, namely those whose unary operations are precisely those coming 
  from $C$.  These can also be characterised as those for which the square
  constituting the unique substitude morphism to the terminal substitude
  $1 \to \mathrm{Comm}$ is a pullback:
  $$\xymatrix{
     C \ar[r]\ar[d] \drpullback &  P \ar[d] \\
     1 \ar[r] & \mathrm{Comm}.
  }$$
  
  The corresponding sub-$2$-category in $\kat{RPat}$ is the 
  $2$-full sub-$2$-category spanned
  by the regular patterns $\freesym C \to M$ which are pullbacks of the
  terminal regular pattern $\freesym 1 \to \kat{FinSet}$.
  In detail, for any (strictified) regular pattern, which in virtue of
  Theorem~\ref{thm:Im(F)} is of the form $\freesym C \to \LHerm P$, we have 
  a commutative square of
  symmetric strong monoidal functors
  $$\xymatrix{
     \freesym C \ar[r]\ar[d] & \LHerm P \ar[d] \\
     \freesym 1 \ar[r] & \kat{FinSet} ,
  }$$
  which clearly is a pullback precisely when the previous square is.
  
  The normality condition plays an important role in the
  theory of generalised multicategories of Cruttwell and
  Shulman~\cite{Cruttwell-Shulman:0907.2460}, a more general framework
  which includes substitudes as a special case.  Their generalised
  multicategories are monoids in $T$-spans, for $T$ a monad on a
  virtual equipment (a particular kind of double category).  They are
  called normalised when the unary operations coincide with the
  vertical arrows in the double category, and the importance of the
  condition is that non-normalised generalised multicategories can be
  reinterpreted as normalised ones, by adjusting the monad and the
  ambient virtual equipment.  Cruttwell and Shulman also consider the
  discrete-objects case, but do not consider the intermediate
  invertible case, which in the present situation is the most
  interesting.
\end{blanko}

\begin{blanko}{Algebras for substitudes.}
Given a substitude $\phi: C\to P$, and a symmetric monoidal category $W$ an algebra consists of
\begin{shortbulletlist}
\item a functor $A : C \to W$
\item for each $\mathbf{x} = (x_1,\ldots,x_n)$ and $y \in C$, an action map
\[ P(\mathbf{x}; y) \longrightarrow W(A(x_1) \tensor \cdots \tensor A(x_n),A(y)) \]
\end{shortbulletlist}
satisfying some conditions. One of these conditions says that for every arrow $f:x \to y$ in $C$, the action of the unary operation $\phi(f)$ is equal to
$A(f): A(x) \to A(y)$.  The remaining conditions are those of an algebra for the operad $P$; in fact, the first condition implies that functoriality in $C$ follows from the $P$-algebra axioms, so to give an algebra for $\phi: C \to P$ in $W$ amounts just
to give an algebra for the operad $P$ in $W$.
  
Just as in the case of operads, the notion of algebra can also be described in
terms of a substitude of endomorphisms: the {\em endomorphism substitude}
$E^A$ of a functor $A: C \to W$ is given in elementary terms as the substitude with category $C$ and $E^A(x_1,\ldots,x_n;y) = W(A(x_1) \tensor \cdots \tensor A(x_n), A(y))$.  An algebra structure on $A: C \to W$ is now the same thing as an identity-on-$C$ substitude map $\phi \to E^A$, that is, an operad map $P \to E^A$ under $C$.  The endomorphism substitude can be described more conceptually as
\[ C \to \End(W)|C , \]
so altogether the notion of algebra is conveniently formulated in 
terms of the substitude adjunction: the endomorphism substitude of 
$A: C \to W$ is precisely $\iopEnd(\tau^A)$ where $\tau^A : \freesym 
C \to W$ is the tautological factorisation of $A$ through $\freesym 
C$, and an algebra is an identity-on-$C$ substitude map $\phi \to 
\iopEnd(\tau^A)$. By adjunction, this is the same thing as an 
identity-on-$C$ pinned symmetric strong monoidal functor
\[ \iopLHerm (\phi) \to \tau^A. \]
Giving this amounts just to giving $\alpha:\LHerm P \to W$ (a 
symmetric strong monoidal functor), that is, an algebra for
the regular pattern corresponding to $\phi$.
  
This works also for algebra homomorphisms: a homomorphism of substitude algebras is a natural transformation $u:A \Rightarrow B$ compatible with the action maps. In terms of endomorphism substitudes, this compatibility amounts to commutativity of the diagram (of operads under $C$)
  \begin{equation}\label{eq:alg-homo}
    \begin{gathered}
    \xymatrix  @R=1pc @C=1.6pc {
     & E^A  \ar[rd]^{u(-,*)} & \\
     P \ar[ru] \ar[rd] && E^{A,B} \\
     & E^B \ar[ru]_{u(*,-)} &
    }\end{gathered}
  \end{equation}
where $E^{A,B}$ is the $W$-bimodule $\freesym C \op \times C \to W$ given by $E^{A,B}(\mathbf{x};y) = W(A(x_1) \tensor \cdots \tensor A(x_n), B(y))$, $u(-,*)$ is induced by postcomposition with $u_y$, and $u(*,-)$ is induced by precomposition with $u_{x_1} \tensor \cdots \tensor u_{x_n}$. Now the natural transformation $u : A \Rightarrow B$ extends tautologically to a monoidal natural transformation $\tau^u : \tau^A \Rightarrow \tau^B$, which in turn extends to a monoidal natural transformation
\[ \xymatrix{
     \LHerm P \ar@/^0.6pc/[r]^\alpha\ar@/_0.6pc/[r]_\beta  \ar@{}[r]|{\Downarrow} 
     &W ,
  } \]
which is the corresponding homomorphism of regular-pattern algebras. The components of this monoidal natural transformation are the same as those of $\tau^u$, since $\freesym C$ and $\LHerm P$ have the same object set; naturality in arrows in $\LHerm P$ is a consequence of \eqref{eq:alg-homo}, and monoidality follows from construction.
  
Clearly this construction can be reversed, to construct an substitude-algebra homomorphism
from a monoidal natural transformation. Altogether we obtain the following proposition.
\end{blanko}

\begin{prop}\label{prop:algebras}
Algebras for a substitude are the same thing as algebras for the corresponding
regular pattern. More precisely for each symmetric monoidal category $W$,
there is an equivalence of categories between the category of algebras
in $W$ for a substitude, and the category of algebras in $W$ for the
corresponding regular pattern. \qed
\end{prop}

Since algebras for a substitude $\phi: C \to P$ are just algebras for the
operad $P$, we immediately get also, via the embedding $\Opd \to \Subst$ by
canonical groupoid pinning:
\begin{cor}
Algebras for an operad are the same thing as algebras for the corresponding
Feynman category.\qed
\end{cor}

\appendix
\setcounter{section}{1}
\addcontentsline{toc}{section}{Appendix: Power coherence}

\section*{Appendix A: Power coherence}


In this appendix we recall coherence for symmetric monoidal categories, from the
point of view of Power's general approach to coherence
\cite{Power-GeneralCoherenceResult}.  This point of view takes as input the
$2$-monad $\freesym$ on $\Cat$ for symmetric monoidal categories, and produces the
coherence theorem for symmetric monoidal categories.  The formulation of this
result given in Lemma~\ref{lem:Power-coherence} below, is most convenient for us
for the purposes of studying regular patterns and Feynman categories.

\begin{blanko}{$2$-monads and their algebras.}
  A \emph{$2$-monad} $T$ on a $2$-category $\ca K$ is just a monad in the sense of
  $\Cat$-enriched category theory.  Thus one has the usual data of a monad
  \[ 
  \begin{array}{lcccr} {T : \ca K \longrightarrow \ca K} && 
  {\eta : 1_{\ca K} \longrightarrow T} &&
  {\mu : T^2 \longrightarrow T} \end{array} 
  \]
  but where $T$ is a $2$-functor, and $\eta$ and $\mu$ are $2$-natural, and the
  axioms are written down exactly as before.  The extra feature is that now one
  has several different types of algebras, and several different types of
  algebra morphisms, and thus a variety of alternative $2$-categories of algebras.
  For instance a \emph{pseudo $T$-algebra} structure on $A \in \ca K$ consists
  of the data of an action $a : TA \to A$, as well as invertible coherence
  $2$-cells $\overline{a}_0:1_A \to a\eta_A$ and $\overline{a}_2:aT(a) \to
  a\mu_A$, which satisfy:
  \[ 
  \xygraph{!{0;(4.5,0):}
  *{\xybox{\xygraph{!{0;(2,0):(0,.5)::} {a}="l" [r] {a\eta_Aa}="m" [d] 
  {a}="r" "l":"m"^-{\overline{a}_0a}:"r"^-{\overline{a}_2\eta_{TA}}:@{<-}"l"^-{\id} 
  "l" [d(.35)r(.7)] {\scriptsize{=}}}}}
  [r]
  *!(0,.02){\xybox{\xygraph{!{0;(2.5,0):(0,.4)::} {aT(a)T^2(a)}="tl" [r] {a\mu_AT^2(a)}="tr" [d] {a\mu_A\mu_{TA}}="br" [l] {aT(a)T(\mu_A)}="bl" "tl":"tr"^-{\overline{a}_2T^2(a)}:"br"^-{\overline{a}_2\mu_{TA}}:@{<-}"bl"^-{\overline{a}_2T(\mu_A)}:@{<-}"tl"^-{aT(\overline{a}_2)} "tl" [d(.5)r(.5)] {\scriptsize{=}}}}}
  [r]
  *{\xybox{\xygraph{!{0;(2,0):(0,.5)::} {a}="l" [l] {aT(a)T(\eta_A)}="m" [d] 
  {a}="r" "l":"m"_-{aT(\overline{a}_0)}:"r"_-{\overline{a}_2T(\eta_A)}:@{<-}"l"_-{\id}
  "l" [d(.35)l(.7)] {\scriptsize{=}}}}}} 
  \]
  When these coherence isomorphisms are identities, we have a \emph{strict}
  $T$-algebra on $A$.  There are also algebra types in which the coherence
  $2$-cells are non-invertible, but these are not important for us here.

  A \emph{lax morphism} $(A,a) \to (B,b)$ between pseudo $T$-algebras is a pair $(f,\overline{f})$, where $f:A \to B$ and $\overline{f}:bT(f) \to fa$, satisfying the following axioms:
\[ \xygraph{{\xybox{\xygraph{{f}="l" [dl] {bT(f)\eta_A}="m" [r(2)] {fa\eta_A}="r" "l":"m"_-{\overline{b}_0f}:"r"_-{\overline{f}\eta_A}:@{<-}"l"_-{f\overline{a}_0} "l" [d(.6)] {\scriptsize{=}}}}}
[r(5)]
{\xybox{\xygraph{!{0;(1.25,0):(0,.8)::} {bT(b)T^2(f)}="p1" [r(2)] {b\mu_BT^2(f)}="p2" [d] {fa\mu_A}="p3" [dl] {faT(a)}="p4" [ul] {bT(fa)}="p5" "p1":"p2"^-{\overline{b}_2T^2(f)}:"p3"^-{\overline{f}\mu_A}:@{<-}"p4"^-{f\overline{a}_2}:@{<-}"p5"^-{\overline{f}T(a)}:@{<-}"p1"^-{bT(\overline{f})} "p1" [d(.8)r] {\scriptsize{=}}}}}} \]
  When $\overline{f}$ is an isomorphism, $f$ is said to be a \emph{pseudo
  morphism}, and when $\overline{f}$ is an identity, $f$ is said to be
  \emph{strict}.  Given lax $T$-algebra morphisms $f$ and $g:(A,a) \to (B,b)$, a
  $T$-algebra $2$-cell $f \to g$ is a $2$-cell $\phi:f \to g$ in $\ca K$ such that
  $\overline{g}(bT(\phi))=(\phi a)\overline{f}$.  In the case $T = \freesym$ one
  has
\begin{shortbulletlist}
\item lax morphism $=$ symmetric lax monoidal functor,
\item pseudo morphism $=$ symmetric strong
monoidal functor,
\item strict morphism $=$ symmetric strict monoidal functor, and
\item algebra $2$-cell $=$ monoidal natural transformation.
\end{shortbulletlist}
The standard notations for some of the various $2$-categories of algebras of a $2$-monad $T$ are
\begin{shortbulletlist}
\item $\PsAlg T$: objects are pseudo $T$-algebras, morphisms are pseudo morphisms,
\item $\Algs T$: objects and morphisms are strict, and
\item $\Algl T$: objects are strict and morphisms are lax,
\end{shortbulletlist}
\end{blanko}

\begin{blanko}{The free symmetric-monoidal-category $2$-monad.}
  The $2$-monad $\freesym$ is described explicitly in Section 5.1 of
  \cite{Weber-PolynomialFunctors}.  For a category $C$, an object of
  $\freesym C$ is a finite sequence of objects of $C$, and a morphism is
  of the form
  \[ (\rho,(f_i)_{1{\leq}i{\leq}n}) : (x_i)_{1{\leq}i{\leq}n} 
  \longrightarrow (y_i)_{1{\leq}i{\leq}n} 
  \]
  where $\rho \in \Sigma_n$ is a permutation, and for $i \in \underline{n} =
  \{1,...,n\}$, $f_i : x_i \to y_{\rho i}$.  Intuitively such a morphism is a
  permutation labelled by the arrows of $C$, as in
  \[
  \xygraph{!{0;(1.25,0):(0,1.25)::} {x_1}="t1" [r] {x_2}="t2" [r] {x_3}="t3" [r] {x_4}="t4"
  "t1" [d] {y_1}="b1" [r] {y_2}="b2" [r] {y_3}="b3" [r] {y_4.}="b4"
  "t1"-"b4"|(.43)@{>}^(.38){f_1} "t2"-@/^{.5pc}/"b1"|(.6)@{>}_(.6){f_2} "t3"-"b3"|(.57)@{>}^(.5){f_3} "t4"-@/_{.5pc}/"b2"|(.57)@{>}_(.54){f_4}}
  \]
  The unit $\eta_C : C \to \freesym C$ is given by the inclusion
  of sequences of objects of length $1$.  The multiplication $\mu_C :
  \freesym^2 C \to \freesym C$ is given on objects by concatenation, and
  on arrows by the substitution of labelled permutations.  Given a symmetric
  monoidal category $M$, $(X_i)_i \mapsto \bigotimes_i X_i$ is the effect on
  objects of a functor $\bigotimes : \freesym M \to M$, whose arrow map
  is described using the symmetries of $M$.  This is the action of a pseudo
  $\freesym$-algebra structure on $M$.  The unit coherences (those denoted
  $\overline{a}_0$ in the general definition) in this case are identities, and
  the components of $\overline{a}_2$ come from the associators and unit
  coherences of $M$.  The strictness of $M$ as a symmetric monoidal
  category, is the same thing as its strictness as an $\freesym$-algebra.  In
  fact, aside from the strictness of the unit coherences, and the specification
  of the tensor product as an $n$-ary tensor product for all $n$, rather than
  just the cases $n = 0$ (the unit usually written as $I$) and $n = 2$ (the
  usual binary tensor product $(A,B) \mapsto A \tensor B$), a pseudo
  $\freesym$-algebra is \emph{exactly} a symmetric monoidal category in the usual
  sense.
\end{blanko}

The coherence theorem for symmetric monoidal categories says that every
symmetric monoidal category is equivalent to a strict one.  Mac Lane's original
proof of this involved a detailed combinatorial analysis.  However from the
point of view of Power's general approach~\cite{Power-GeneralCoherenceResult}, this
result comes out of how $\freesym$ interacts with the \emph{Gabriel
factorisation system} on $\Cat$, by which every functor $f : A \to B$ is
factored as
\[ 
\xygraph{{A}="p0" [r] {C}="p1" [r] {B}="p2" "p0":"p1"^-{g}:"p2"^-{h}}
\]
where $g$ is bijective on objects, and $h$ is fully faithful.

\begin{blanko}{The Gabriel factorisation system.}
  The Gabriel factorisation system, in which the left class is that of 
  bijective-on-objects functors, and the right class consists of the fully 
  faithful functors, is an orthogonal factorisation system, meaning that 
  arrows in the left class admit a unique lifting property with respect to 
  arrows in the right class. One way to formulate this, is to say that for 
  any bijective-on-objects functor $b : A \to B$, and fully faithful functor 
  $f : C \to D$, the square 
  \[ {\xygraph{!{0;(2.5,0):(0,.4)::}
  {\Cat(B,C)}="p0" [r] {\Cat(A,C)}="p1" [d] {\Cat(A,D)}="p2" [l] 
  {\Cat(B,D)}="p3" "p0":"p1"^-{\Cat(b,C)}:"p2"^-{\Cat(A,f)}:@{<-}"p3"^-{\Cat(b,D)}:@{<-}"p0"^-{\Cat(B,f)}}
  } \]
  is a pullback in the category of sets.  In fact the Gabriel factorisation is
  enriched over $\Cat$, meaning that these squares are also pullbacks in
  $\Cat$.  Explicitly on arrows this says that given $\alpha$ and $\beta$ as in
  
  
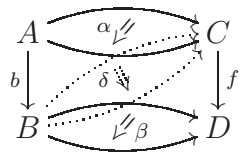
\vspace*{-24pt}
  
  \[ 
  \xygraph{!{0;(2,0):(0,.5)::} {A}="p0" [r] {C}="p1" [d] {D}="p2" [l] {B}="p3" "p0":@/^{.75pc}/"p1"|(.6){}="tpd":"p2"^-{f}:@/_{.75pc}/@{<-}"p3"|(.4){}="bpd":@{<-}"p0"^-{b}
  "p0":@/_{.75pc}/"p1"|(.4){}="tpc" "p3":@/_{.75pc}/"p2"|(.4){}="bpc"
  "tpd":@{}"tpc"|(.25){}="td"|(.75){}="tc" "td":@{=>}"tc"|-{}="ta"
  "ta" [u(.05)l(.1)]{\scriptstyle{\alpha}}
  "bpd":@{}"bpc"|(.25){}="bd"|(.75){}="bc" "bd":@{=>}"bc"|-{}="ba"
  "ba" [d(.05)r(.1)]{\scriptstyle{\beta}}
  "p3":@{.>}@/^{.75pc}/"p1"|(.5){}="mpd" "p3":@{.>}@/_{.75pc}/"p1"|(.5){}="mpc"
  "mpd":@{}"mpc"|(.25){}="md"|(.75){}="mc" "md":@{:>}"mc"|-{}="ma"
  "ma" [l(.1)]{\scriptstyle{\delta}}} 
  \]
  
  \vspace*{-12pt}
  
\noindent  such that $f\alpha = \beta b$, there exists a unique
  $\delta$ as shown such that $\alpha = \delta b$ and $f\delta = \beta$.  Moreover
  the Gabriel factorisation system is an \emph{enhanced} factorisation system in
  the sense of Kelly~\cite{Kelly:enhanced}, which is to say that given $u$, $v$ and $\alpha$ as on the left,
  \[ 
    \xygraph{{\xybox{\xygraph{!{0;(2,0):(0,.5)::} {A}="p0" [r] {C}="p1" [d] {D}="p2" [l] {B}="p3" "p0":"p1"^-{u}:"p2"^-{f}:@{<-}"p3"^-{v}:@{<-}"p0"^-{b}:@{}"p2"|-*{\iso}="m" "m" [u(.2)] {\scriptstyle{\alpha}}}}}
    [r(5)]
    {\xybox{\xygraph{!{0;(2,0):(0,.5)::} {A}="p0" [r] {C}="p1" [d] {D}="p2" [l] {B}="p3" "p0":"p1"^-{u}:"p2"^-{f}:@{<-}"p3"^-{v}:@{<-}"p0"^-{b}:@{}"p2"
    "p3":"p1"^(.4){d}|-{}="m" "m":@{}"p2"|(.45)*{\iso}="mb" "mb" [u(.2)] {\scriptstyle{\beta}}}}}} 
  \]
  there exists unique $d$ and $\beta$ as shown on the right, such that $u = db$
  and $\alpha = \beta b$.  It is not difficult to verify that the Gabriel factorisation
  system enjoys these properties~\cite{StreetWalters-YonedaStructures}.
\end{blanko}

From the explicit description of $\freesym$ it is easy to see that $\freesym$
preserves bijective-on-objects functors.

With these details in hand Power's idea boils down to the following.  Given a
symmetric strong monoidal functor $F : S \to M$ where $S$ is a
symmetric {\em strict} monoidal category, one can take the Gabriel factorisation
\[ 
\xygraph{{S}="p0" [r] {M'}="p1" [r] {M}="p2" "p0":"p1"^-{G}:"p2"^-{H}}
\]
of $F$, and then factor the coherence witnessing $F$ as strong monoidal on the left
\[
  \xygraph{{\xybox{\xygraph{!{0;(2,0):(0,.5)::}
  {\freesym S}="p0" [r] {S}="p1" [d] {M}="p2" [l] {\freesym M}="p3" "p0":"p1"^-{\bigotimes}:"p2"^-{F}:@{<-}"p3"^-{\bigotimes}:@{<-}"p0"^-{\freesym(F)}
  "p0" [d(.6)r(.5)] {\iso} [u(.25)] {\scriptstyle{\overline{F}}}}}}
  [r(2.5)] {=} [r(2.5)]
  {\xybox{\xygraph{!{0;(2,0):(0,.5)::} {\freesym S}="p0" [r] {S}="p1" [d] 
  {M'}="p2" [d] {M}="p3" [l] {\freesym M}="p4" [u] {\freesym M'}="p5" "p0":"p1"^-{\bigotimes}:"p2"^-{G}:"p3"^-{H}:@{<-}"p4"^-{\bigotimes}:@{<-}"p5"^-{\freesym(H)}:@{<-}"p0"^-{\freesym(G)} "p5":@{.>}"p2"^{\bigotimes}
  "p0" [d(.45)r(.5)] {=}
  "p5" [d(.6)r(.5)] {\iso} [u(.25)] {\scriptstyle{\overline{H}}}}}}} 
\]
uniquely as on the right, using enhancedness and the fact that $\freesym G$ is
bijective on objects.
\begin{lemma}\label{lem:Power-coherence} {\upshape \cite{Power-GeneralCoherenceResult}}
  In the situation just described, $\bigotimes : \freesym M' \to M'$ is a
  symmetric strict monoidal structure on $M'$.  With respect to this
  structure, $G$ is a symmetric strict monoidal functor, and $\overline{H}$ is
  the coherence datum making $H$ into a symmetric strong monoidal functor.
  \qed
\end{lemma}

We attribute this result to Power, although it is not exactly formulated in this
way in \cite{Power-GeneralCoherenceResult}.  Power's result applies to more
general monads $T$ than $\freesym$ (only required to preserved
bijective-on-objects functors), but for the functor $F$ he only considers the
case where $S = T M$ and $F$ is the action $T M \to M$.  It is easy to adjust
his argument to the present situation, to verify the strict $\freesym$-algebra
axioms for $\bigotimes : \freesym M' \to M'$, and the pseudo $\freesym$-morphism
axioms for $\overline{H}$.  Note that $G$ is a strict $\freesym$-algebra
morphism by construction.

\end{document}